\documentclass[letter,11pt]{amsart}
\usepackage{verbatim}
\usepackage{amsmath}
\usepackage{amssymb, latexsym}
\usepackage{amsthm}
\usepackage{amscd}
\usepackage{amsfonts,mathrsfs}
\usepackage{delimseasy1.01}
\newcounter{keepenumi}
\newcommand{\enumiget}{\setcounter{enumi}{\value{keepenumi}}}
\newcommand{\enumisave}{\setcounter{keepenumi}{\value{enumi}}}%
\newcounter{keepenumii}

\newtheorem{theorem}{Theorem}

\newtheorem{lemma}[theorem]{Lemma}

\numberwithin{theorem}{section}
\newtheorem{cor}[theorem]{Corollary}
\newtheorem{lem}[theorem]{Lemma}
\newtheorem{prop}[theorem]{Proposition}
\newtheorem{thm}[theorem]{Theorem}
\newtheorem{definition}[theorem]{Definition}
\theoremstyle{remark}

\newtheorem{remark}[theorem]{Remark}

\newtheorem{rem}[theorem]{Remark}
\newtheorem{rems}[theorem]{Remarks}
\newtheorem{conj}[theorem]{Conjecture}
\newtheorem{ques}[theorem]{Question}
\numberwithin{equation}{section}
\newcommand{\conjref}[1]{Conjecture~\ref{#1}}	
\newcommand{\corref}[1]{Corollary~\ref{#1}}	
	
\newcommand{\itemref}{\refitem}	
\newcommand{\lemref}[1]{Lemma~\ref{#1}}	
\newcommand{\propref}[1]{Proposition~\ref{#1}}	
	
\newcommand{\refitem}[1]{(\ref{#1})} 	
\newcommand{\remref}[1]{Remark \ref{#1}}	
	
\newcommand{\secref}[1]{\S\ref{#1}}	
\newcommand{\thmref}[1]{Theorem~\ref{#1}}	
\newcommand\itref\itemref	
	
\newcommand\refthm\thmref		
\newcommand{\pii}{\pi i}	
\newcommand\R{{\mathbb R}}

\newcommand\supp{\text{Supp\	}}
\newcommand\Supp{\text{Supp\	}}
\newcommand{\T}{\text{$\mathbb T$}}
\newcommand\Times{\text{$\times$}} 
\newcommand{\Z}{\text{$\mathbb Z$}}	
\newcommand\Zp[2]{\texttt{$\mathbb Z_{#1}^{\	#2}$}}
\newcommand\setof[2]{\{\,{#1}\,:\, {#2}\,\}}
\newcommand\mpar[1]{\marginpar{\tiny #1}} 
\newcommand\omitpf{\omitproof}	
\newcommand\omitproof[1]{} 
\newcommand\proofBitFormat[1]{\textit{#1}}

\title[A beastiary\dots -- \today]{A beastiary of sets having extremal Sidon constant, or,  there must be more
than one theorem somewhere here \\ \  \today  }
\author[Graham \& Ramsey-- \today]{Colin C. Graham}
\address{\#208 -- 6100 Sixth Ave., Whitehorse YT Y1A 1M5, Canada}
\author{L. Thomas Ramsey}
\address{Department of Mathematics\\University of Hawaii at Manoa, Honolulu HA, USA}

\subjclass[2000]%
{Primary: 42A55, 42A63, 43A25, 43A46; Secondary: 43A05, 43A25}


\thanks{This work began with an equipment grant from NSERC in 1990 when I was at Lakehead University. I thank both NSERC and Lakehead.
The Macintosh II acquired then was not able to give many interesting results
 and no publications resulted, though a Mac program was completed and distributed
  to a few people. 
  }
\begin{document}
\begin{abstract}
New sets (typically found by computer search)  with Sidon constant equal to the square root of their cardinalities are given. For each  integer $N$ there are
only a finite number of groups of prime order containing $N$-element extreme sets. Some extreme sets
appear to fit a pattern; others do not.  Various conjectures and questions are given.
\end{abstract}

\maketitle

\tableofcontents

\section{Introduction} \label{secintro}
Let $G$ be a locally compact abelian group with dual group $\widehat G$. Haar measure on $G$ will be
denoted by $m_G$ and is counting measure if $G$ is discrete. 
\mpar{Removed all marginal comments earlier than 2019-08-06 \\
as of 2019-08-07}

The \emph{Sidon (or Helson) constant}, $S(E)$, of a set $E\subset G$ is 
the infimum of constants,
$C$, such that $\|\mu\|\le C\|\widehat\mu\|_\infty$ for all
 (non-zero) measures $\mu$ concentrated
on $E$. Here $\widehat\mu$ 
is the Fourier(-Stieltjes) transform of $\mu$. 
The Sidon constant, $S(E)$, is always at least $\sqrt{\#E}$.

\begin{definition} A non-zero measure $\mu$ on the discrete abelian group $G$ is \emph{extreme} 
(or \emph{extremal}) 
if $\|\widehat\mu\|_\infty = \|\mu\|/\sqrt{\#\supp\mu}$,
 where $\supp\mu$ is the support of $\mu$.
A finite set $E$ is \emph{extreme} (or \emph{extremal}) 
if it is the support of an extreme measure.
If $\mu$ is an extreme measure with support $E$, we shall say $\mu$ is \emph{extreme for} $E$.
\end{definition}

The definition of ``extreme'' includes the finiteness of the set.
Finite abelian groups are extreme \cite{MR0458059}. Other extreme sets were given in \cite{MR627683}. The 
contributions here are: 
\begin{itemize}

\item 
more examples of extreme
 sets but with most of the (mostly both tedious and obvious) verifications left to the 
 appendix.
 
 \item 
 a proof that for each integer $N$ 
 there are only a finite number of groups of prime
 order that contain $N$-element extreme sets (\thmref{thmLimExtremeSetPrimeProds});
 
 \item 
 some conjectures and questions.
 
 

\end{itemize}
Throughout this paper, $G$ will be a discrete abelian group (almost always finite) with counting measure
and $\Gamma$ the dual of $G$ with Haar measure $m_\Gamma(\Gamma) = 1$. So here the
Plancherel theorem says $\sum_{g\in G}|\mu(\{g\})|^2 
= \int_\Gamma|\widehat\mu(\gamma)|^2dm_{\Gamma}.$ 
When $G$ is finite,   the
Plancherel theorem  becomes 
\[
\sum_{g\in G}|\mu(\{g\})|^2 = (\#G)^{-1}\sum_{\gamma\in\Gamma}|\widehat\mu(\gamma)|^2.
\]
The set  of (regular, Borel) measures on a set $E\subset G$ will be 
denoted by $M(E)$.

\bigbreak
\subsection*{Organization of this paper} General properties of extreme sets are covered in 
\secref{secTheory}. That includes conditions necessary for extremality (several) and sufficient conditions. The section
concludes with \thmref{thmLimExtremeSetPrimeProds}, which states that
 for each $N$ there is a bound on primes $p_1, p_K$ such that an $N$-element set is
contained in the product of $K$ cyclic groups of respective order $p_1, \dots, p_K$.
 
Our ``Beastiary'' is in \secref{secBeastiary}: tables of extreme sets with sample extreme measures, along with commentary
on those sets.

Section \ref{secConjectures} contains conjectures, questions, and proofs of some of the results summarized in the 
tables of the
previous section.

The final section discusses the computer programs used in the project.

\section{Theory}\label{secTheory}

\subsection{Properties of extreme measures}

\begin{lemma}\label{lemExtrHatConstant}
Let $\mu$ be an extreme measure on the discrete abelian group $G$. Then $|\widehat\mu|$ is constant.
\end{lemma}

\begin{proof} See
\cite[Lemma 1.1]{MR627683} for a different proof.
Let $E$ be the support of $\mu$. We may assume $\|\mu\|=N=\#E$. Then $\|\widehat\mu\|_\infty = \sqrt N$.
If there is $\gamma\in \Gamma$ such that $|\widehat\mu(\gamma)|<\sqrt N$, then
there is an open neighbourhood $U$ of $\gamma$ such that $|\widehat\mu|<\sqrt N$ on $U$. Thus,\
\[
\int_\Gamma|\widehat\mu(\gamma)|^2=\Big(\int_U+\int_{\Gamma\backslash U}\Big)|\widehat\mu|^2d\gamma<
Nm_\Gamma(U)+Nm_\Gamma(\Gamma\backslash U)<N.
\]
The Plancherel theorem now tells us that $\sum_{g\in E}|\mu(\{g\})|^2<N.$ We 
use Jensen's inequality for $\varphi (x) = x^2$, weights $a_g$ and values $y_g$:
\[
\Big(\frac{\sum a_g y_g}{\sum a_g}\Big)^2\le \frac{\sum a_g y_g^2}{\sum a_g}.
\]
Set $a_g=1/N$    and $y_g = |\mu(\{g\})|$ for $g\in E$ . Then $\sum a_g=1$ and 
\begin{align*}
N^{-2}\|\mu\|^2
&=\Big(\frac{N^{-1}\sum_g |\mu(\{g\})|}{NN^{-1}}\Big)^2 
\ \overset{\text{\tiny Jensen}}{\le\phantom{|}}\  \frac{N^{-1}|\sum_g \mu(\{g\})|^2}{1}
< N^{-1}N 
=
1,
\end{align*}
so $\|\mu\|^2<N^2$, contradicting the assumption that $\|\mu\|=N$.
\end{proof}

The direction (1) $\Rightarrow$ (2) of following result 
is \cite[Lemma 1.1]{MR627683}; the other
 direction must be known but we are unable to give a reference.
\begin{theorem}\label{thmExtrMsConstant} Let $\mu$ be a discrete measure on an abelian group. Then the following are equivalent.
\begin{enumerate}
  \item
$\mu$ is an extreme measure. 
\item 
 $|\mu(x)|$ is constant on $\Supp\mu$ and
$|\widehat\mu|$ is constant on the dual group.
\end{enumerate}
\end{theorem} 

\begin{proof}[Proof of \thmref{thmExtrMsConstant}]
(1) $\Rightarrow$ (2). See
\cite[Lemma 1.1]{MR627683} for a different proof.

Let $\mu$ have support $E$.  We may assume that
$\|\mu\|=N=\#E$ and $|\widehat\mu|=\sqrt{N}$ on $\Gamma$. For $g\in E$ let $\varepsilon_g=1-|\mu(\{g\})|$.
Then $N=\sum_{g\in E}|\mu(\{g\})| = \sum_{g\in E}(1-\varepsilon_g)$,
so $\sum\varepsilon_g=0$.

Let $f=\frac{d\mu}{dm_G}$, so 
$\int hd\mu = \int hf\,dm_G$ for all $h:G\to\mathbb C$. 
Then 
\begin{align*}
\int |f|^2dm_G &= \sum_{g\in E}|f(g)|^2 
=\sum_{g\in E}(1-\varepsilon_g)^2
\\
&=\sum_{g\in E}(1-2\varepsilon_g+\varepsilon_g^2)
={N} +N\sum\varepsilon_g^2.
\end{align*}
Hence, 
\begin{equation}\label{eqL2ofmu}
\sum_g|\varepsilon_g|>0 \text{ implies }
 \|f\|_2>\sqrt{{N}}.
 \end{equation}
  On the other hand, $\widehat f= \widehat\mu$ on $\Gamma$ and
so (using Plancherel and extremality)
\begin{align}
{N}+N\sum\varepsilon_g^2 =\int |f|^2dm_G &= 
\int|\widehat f|^2 dm_{\Gamma}=
\int_\Gamma|\widehat\mu|^2 dm_\Gamma
%
= N,
\end{align}
so the $\varepsilon_g$ are all zero, by \eqref{eqL2ofmu}. 
That proves that $|\mu(\{g\})|\equiv 1$
on $E$. 


(2) $\Rightarrow$ (1).
Suppose $|\mu|$ is constant on its support, $E$, 
and that $|\widehat\mu|$ is constant on $\Gamma$.
We may 
assume
that  $|\mu|=1$ everywhere on $E$. 
 As before, let $f=\frac{d\mu}{dm_G}$
 so $f(g)=\mu(\{g\})$ for all $g \in G$ and $\widehat\mu=\widehat f$. Then 
\begin{align}
\|\mu\|=N,  \int|f|^2dm_G= N = \int \|\widehat\mu\|_\infty^2 dm_\Gamma =
\|\widehat\mu\|_\infty^2.
\end{align}
Hence $N = \|\widehat\mu\|_\infty^2$ and $\|\widehat\mu\|_\infty =\sqrt N$, so $\mu$ is extremal.
\end{proof}

The Pontriagin duality theorem and \thmref{thmExtrMsConstant} immediately yield:

\begin{cor}\label{corHatIsExtreme}
If $\mu$ is extreme for the finite abelian group $G$, then
 $
\widehat\mu\, m_{\widehat G}$ is
extreme for $\widehat G$.

\end{cor}
\medbreak
The following corollary helps eliminate candidate sets for extremality and has been a useful tool in
the search for small   extreme  subsets of larger groups.  
 
\begin{cor}\label{corExtrCnxln}
\mpar{\ \\ Statement and proof revised. \\2019-08-07}
Let $E$ be finite with $\#E>1$. Let $\nu= \sum_{x\in E}\delta_g$. If $E$ is the support of a non-zero measure $\mu$
such that $|\widehat\mu|$ is constant  on $\Gamma$ (in particular, if $E$ is extreme),
 then every point mass in $\nu*\tilde \nu$ has coefficient either at least 2 or 0.
\end{cor}
\begin{proof}
Because $|\widehat\mu|$ is constant, $\widehat\mu \,\overline{\widehat\mu}$ is constant and
so 
\[
\mu*\widetilde\mu=|\widehat\mu(0)|^2\delta_0=\sum_g|\mu(\{g\})|^2.
\]
 Hence,
for every $g\in (E-E)\backslash\{0\}$, $\mu*\tilde \mu(\{g\})=0$, so there must be cancellation of $\mu*\widetilde\mu$'s masses at $g$, which requires
that $\nu*\tilde\nu(\{g\})>1$.
\end{proof}

\begin{lem}\label{lemNoExtrInZn}
Let $M\ge 1$. Suppose $E$ is a compact subset of the Hilbert space $\mathcal H$
and has   two or more
elements. 
Then  there is an element of $E-E$ that has only one
representation as a difference of two elements of $E$.  
\end{lem}

\begin{proof}
Note that the hypotheses and conclusion of the lemma hold for $E$ if and only if they hold for any translate of $E$ if and only if they hold for any rescaling of $E$, that is, replacing $E$ 
with $sE=\{sx:x\in E\}$, for any $s>0$.

 Because $E \times E$ is compact, the set $E$ has a finite diameter $D$:
$$
D=\sup\setof{\Vert x-y\Vert}{x, \, y \in E}<\infty.
$$
Here $\Vert\cdot\Vert_{\mathcal H}$ is the norm for $\mathcal H$.

By compactness, there are $u \in E$ and $v \in E$ such that $\Vert u-v \Vert=D$.  
Because $E$ has at least 2 members, $D \ne 0$ and $u \ne v$. 
 By replacing $E$ with $E-u$ if necessary, we may assume $u=\mathbf 0$, the identity
 of $\mathcal H$. 
 Thus, we may also assume $\mathbf 0\in E$.
  By replacing (the possibly translated) $E$ with $\frac 1 D E$,
we may assume $D=1$. Let $\mathcal B$ be an 
orthonormal basis of $\mathcal H$ whose first element is $v$.  We note that
the coordinates of $v$ with respect to $\mathcal B$ are $(1,0,\dots,0)$.

Since $\mathbf 0\in E$, $E\subset E-E$.

For each element $g\in\mathcal H$ and $1\le m\le M$,  let $t(g)$ 
be the first coefficient with respect to $\mathcal B$, that
is, $t(g)= \langle g,v\rangle.$
Since $E$ has diameter 1, $1\ge \|g\|\ge t(g)$ 
for all $g\in E$. Furthermore, $t(g)\ge 0$ for all $g\in E$
since otherwise such a $g$ would have distance greater than 1 from $v$.

Now suppose $v=x-y$ where $x,y\in E.$
Then if $t(y) = 0$, we have $y={\mathbf 0}$, since otherwise $\|y-v\|>1$.
Also, if $t(x) = 1$, then $x=v$, since otherwise $\|g-\mathbf 0\|>1$.

We can now show that $v$ is an element of $E-E$ with only one
representation, $v= v-\mathbf 0$, as   a difference of two elements of $E$.

\begin{enumerate}
\item If $x\in E$ and $y\in E\backslash\{{\mathbf{ 0}}\}$, 
then $t(y)>0$, so
$t(x-y) = t(x) -t(y)\le 1-t(y)<1\ne t(v)$. 
Hence, $x-y\ne v$.
\item  If $x\in E\backslash \{v\}$ and $y\in E$, 
then $t(x-y)= t(x) -t(y) <1-t(y)<1$ so $x-y \ne v$.
\end{enumerate}
It now follows that the only representation of $v$ in $E-E$ is as
$v-\mathbf 0$.
\end{proof}

The following is immediate from \corref{corExtrCnxln} and \lemref{lemNoExtrInZn}.

\begin{cor} \label{corExtremeRm}
Let $1\le M<\infty$. 
\begin{enumerate}
\item \label{itcorExtremeRm1} 
Let $\mu$ be a finitely supported measure on $R^M$ such that
$|\widehat\mu|$ is constant. Then $\mu$ is a point mass.
\item 
The only extreme subsets of $\R^M$ $($hence, of $\Z^M)$  are singletons.
\end{enumerate} 
\end{cor}

A measure $\mu$ such that $|\widehat\mu|$ is constant will be called a
 \emph{transform with constant absolute value} (\emph{TCAV)} measure. Clearly TCAV is 
 equivalent to $\mu*\widetilde\mu=a\delta_0$ for
some $a$. If a set is the support of a TCAV measure it will be called a TCAV set.
We do not know if being TCAV implies that a set is also extreme, as our only
examples distinguising ``extreme'' from ``
TCAV'' applies to the measures in the next remark.

\begin{rem}\label{remTCAVNotExtr}
Computation shows that the measure 
\mpar{This remark is new; its unenlightening proof in the source, commented out.
\\2019-08-18}
$\mu=\delta_0+(1+\sqrt3 i)\delta_1+(1-\sqrt3 i)\delta_2$
is TCAV on \Zp3{}
and $\nu=\delta_0+\varepsilon i\delta_1$ is TCAV on \Zp2{}, for all 
$0\le \varepsilon<\infty$. They are  not extreme. The transform of $\mu$ takes on the values 
$\pm 3$ and that of $\nu$ the values $1\pm i\varepsilon$.
%

The example $\nu$ shows that the limit argument of \lemref{lemLimExtremeSet} can fail for TCAV measures.
\end{rem}
A set that is the sum (product) of two extreme sets is extreme if its cardinality is the product of the
cardinalities of the summands (Proof: take the convolution of an extreme measure on each set).
However, sets of the form $\Z\times F$ where $F$ is a finite group contain extreme sets of cardinality $(\#F)^2$;
such sets cannot be contained in a coset of a finite group; see \propref{propUnionCosets}.

\subsection{Subgroups and quotients}
The following is a simplified version of \cite[Theorem 2.1 (i)]{MR627683}\label{thmGR2.1}.
\begin{prop}\label{propUnionCosets}
Let $H\subset G$ be a subgroup of cardinality $N$. If $g_n+H$, $1\le n\le N$,
are distinct cosets of $H$, then $\bigcup _1^Ng_n+H$ is extreme.
\end{prop}

\begin{proof}
Let $m_H$ be Haar measure on $H$ with $\|m_H\|=N$. 
Let $\lambda_1,\dots,\lambda_N$ be
elements of $\Gamma$ whose restrictions to $H$ are the $N$ characters of $H$.
Then $\widehat{m_H}$ is $N$ times the characteristic function 
of $H^\perp$ and $\widehat\nu_n =\widehat{\lambda_nm_H}$ 
is the $N$ 
times characteristic function of
$\lambda_n+H^\perp$, $1\le n\le N$. Hence
\begin{equation}\label{eqZeroProducts}
\widehat \nu_n\widehat\nu_\ell =0 \ (1\le n\ne\ell\le  N.
\end{equation}
Let $\tau_n=
\delta_{g_n}*
(\lambda_n m_H)$,
 $1\le n\le N$ and $\mu = \sum_1^N\tau_n$.
 Then, 
$|\widehat\mu| \equiv N$ everywhere, the support of $\mu$ is  
the union of $N$ disjoint cosets of
 size $N$, so $\|\mu\|= N^2$ and $\mu$
 is extreme.
\end{proof}


The following was suggested by L.~T.~Ramsey and is included with permission.
 \begin{prop}
 \label{propRamseyProjection}
 Let $H$ be a subgroup of $G$ and $E\subset G$ extreme. If $\#(E/H) =\#E$, then $E/H$ is extreme in $G/H$.
 \end{prop}
 
 \begin{proof}
 Let $\mu$ be an extreme measure on $E$, with $\|\mu\|=\#E$ and $\|\widehat\mu\|_\infty=\sqrt{\#E}$.
 Let $\nu$ be the measure on $E/H$ given by $\nu(g+H)=\mu(\{g\})$ for $g\in E$.
 If $\gamma\in H^\perp=(G/H)\widehat{\ }\subset \widehat G$, then 
 \begin{align*}
 \widehat\nu(\gamma) =
   &\sum_{g+H\in E/H} \langle-\gamma,g+H\rangle\nu(g+H)
   \\
   = &\ \quad\sum_{g\in E}\quad \ \langle-\gamma,g\rangle\mu(\{g\})
   \\
   = &\quad\ \widehat\mu(\gamma).
 \end{align*}
 Thus, $\|\nu\|=\#(E/H)=\#E$ and $\|\widehat\nu\|_\infty=\sqrt{\#(E/H)}.$ Therefore $\nu$ and $E/H$ are extreme.
 \end{proof}
 
 \begin{rem} \label{remRamseyNoLifting}
 The converse is false. Indeed, let $E=\{0,1,2\}\subset G=\Zp{8}{}$ and $H=\{0,4\}.$
Then $\{0,1,2\} $  fails the test of \corref{corExtrCnxln}: 2 and 6 have unique representations in the difference set,
 so $\{0,1,2\}$ is not extreme in \Zp8{}.\footnote{Alternatively, apply 
   \cite[3.1]{MR627683}'s enumeration of three-element extreme sets, or  just use the computer, which quickly
    confirms that $\{0,1,2\}$ is not extreme in \Zp8{}.}
 Now
 let $\tau:G \to G/H$ be the natural mapping. Then $\tau(E)=\{0,1,2\}$ a
 three element set in $\Zp4{}$ and therefore  extreme.
 \end{rem}

\subsection{Automorphisms and equivalent sets}

A group of prime order $p$ has $p-1$ automorphisms, given by multiplication by integers $1\le j<p$.
 Every element (except the identity) is moved by every non-trivial automorphism. 
A group of prime power order $p^k$ also has $p-1$ automorphisms 
\cite[Thm. 4.1]{MR2363058}.\footnote{A proof of this fact is 
sketched in \cite{ABeastiaryAppendix}.}
See \cite{MR2666671, MR2363058, MR2103185, MR1512510}
 for more  on automorphisms of
finite abelian groups. 

Two subsets, $E,F$ of a group $G$ are \emph{equivalent}
 if one of them can be obtained from the other
by a sequence of group automorphisms and translations. 
It is immediate that
all the sets in an equivalence class have 
the same Sidon constant. It is not true that
having the same Sidon constant and cardinality implies two subsets
 are equivalent; examples are provided by
\propref{prop5in12} and the subsets of \Zp2{}\Times\Zp4{}
 given in \cite[3.3 (i)]{MR627683}.
Whether two \emph{non}-extreme sets in different 
 equivalence classes can have the same Sidon
constant is unknown. Our computer program has not 
found any for sets of cardinality up to  7 in
groups of order less than 30.

\bigbreak

In groups of prime order, all two-element sets are equivalent. 
On the other hand, in the group $\Zp 7{}$ there are two 
equivalence classes of three element sets, one
generated by $\{0,1,2\}$ and the other 
by $\{0,1,3\}$. Since equivalent sets have 
 equal Sidon constants, and those two sets have different constants, that gives a different proof that
those two sets are not equivalent.
It is often useful to list the elements of the equivalence classes 
for each finite group and our programs that search for extreme sets do more-or-less that as a
preliminary step.

\subsection{The pseudo-Sidon constant (PSC)}

It is easier to write programs to calculate the infimum
\mpar{Revised defn of PSC to reflect what computer program does.
\\
Don't know if $PSC(E)=1/S(E)$.
\\2019-08-23} \begin{equation}\label{eqPSconstDef} 
PSC(E) = \inf\{ \|\widehat\mu\|_\infty: \text{supp } \mu = E, |\mu(\{x\})|=1
 \ \forall x\in E\}
\end{equation} 
than to compute the Sidon constant directly and select sets for which it is extreme:.
We call $PSC(E)$ the \emph{pseudo-Sidon constant (PS constant, PSC)} of $E$ and note the trivial:
\begin{prop}
$E$ is extremal if and only if $PSC(E)=S(E)$, which occurs if and only if $PSC(E)=\sqrt{\#E}$.
\end{prop}

When we want to be clear about the group which contains $E$, we will write $PSC(E,G)$.

\subsubsection{The PSC and sets with two elements}\label{subsec2elts}

That the PSC can give an interesting (or odd) result is evidenced by
Part 2 of the following  (which does not extend in any obvious way, though the computer 
does suggest the existence of other sets with PSCs of $\sqrt3$  and of others with integer PSCs).
Part (3) is    \cite[3.4(ii)]{MR627683} and Part (1), if new, is also obvious.  

\begin{prop} \label{propTwoEltsinZedthree}
\begin{enumerate}\item
$PSC(\{0,1\},\Zp n{})$ increases monotonically to 2 as $n\to\infty$.
\item   $PSC( \{0, 1 \},\Zp3{})=\sqrt3$. 
\item \label{it2EltCyclic} The only two-element extreme sets are cosets.
\end{enumerate}
\end{prop}

\begin{proof} (1). 
Let $\mu = \delta_0 +e^{2\pi i\theta}\delta_1$ on \Zp{n}{}.
Then 
\begin{equation}\label{eqhatnuTwoElts}
\widehat\mu(k)= 1+e^{2\pi i(\theta-\frac{k}{n})},  \text{ for $k$ in the dual of \Zp{n}{} (i.e., $0\le k<n$).}
\end{equation}
 It is clear from \eqref{eqhatnuTwoElts} that the minimum (over $\theta$) of
 $\sup_k|\widehat\mu(k)|$ will occur  only when $\theta$ is an odd 
 multiple of $\frac{1}{2n}$ and that the minimum will be
 $|1+e^{2\pi i/2n}|$. Clearly, $|1+e^{2\pi i/2n}|\to2$ monotonically.

(2). Let $\nu = \delta_0+e^{2\pi i/3}\delta_1$. It will be clear from the next paragraph that 
$\|\widehat\nu||_\infty = \sqrt3,$
so the PSC is at most $\sqrt3$.

Note that if, more generally, $\mu = \delta_0+ e^{i\theta}\delta_1$, then $|\widehat\mu(0)|^2 =
{2 +2\cos\theta}$. Hence $|\widehat\mu(0)|\le \sqrt3 $ if and only if 
\begin{equation}\label{eqTwoEltSet1}
\frac\pi3\le\theta\le \frac{5\pi}3 \mod 2\pi.
\end{equation}
Now, $\widehat\mu(1) = 1 + e^{(\theta - 2\pi/3)i}$, so $|\widehat\mu(1)|\le \sqrt3$ if and only if 
\begin{equation}\label{eqTwoEltSet2}
-\frac{\pi}3\le \theta\le {\pi} \mod 2\pi.
\end{equation}
 Finally, $
\widehat\mu(2) = 1 + e^{(\theta-4\pi /3)i},
$
so $|\widehat\mu(2)|\le\sqrt3$ if and only if 
\begin{equation}\label{eqTwoEltSet3}
- \pi\le\theta\le\frac{\pi}3 \mod 2\pi.
\end{equation}
Putting \eqref{eqTwoEltSet1}-\eqref{eqTwoEltSet3} together, we see that
$\theta=\pm \frac\pi3$ and  $||\widehat\mu||_\infty = \sqrt3.$ 

(3). \cite[3.4 (ii)]{MR627683}  It is enough to show that $\{0,k\}\subset \Zp{m}{}$ is extreme if and only if $k$ divides $m$ and $m/k=2$.
But if $\nu = \delta_0+\alpha\delta_k$ has $\|\widehat\nu\|_\infty=\sqrt2$ and, in particular, 
$|\widehat\nu(0)|=|1+\alpha|=\sqrt2$, then $\alpha=\pm i$. We may assume $\alpha=i$. Hence, $\widehat\nu(\gamma) = 1
+ i\langle \gamma,k\rangle$ for $\gamma$ in the dual of \Zp{m}{}. Therefore $\langle\gamma, k\rangle=\pm1$ for all $\gamma$. 
Hence, $k$ has order 2, that is, $m/k=2$.

Here is a second proof of (3):
let $\mu=\delta_0+\delta_g$. Then $\mu*\tilde \mu = 2\delta_0+\delta_g +\delta_{-g}$. If $g\ne -g$, then
$E$ fails the test of \corref{corExtrCnxln}. Since subgroups are extreme, the conclusion follows.
 \end{proof}

\begin{rem}
We note that the transform of  $\nu $  in (2) does not
have a constant absolute value, as is to
 be expected from \thmref{thmExtrMsConstant}.
\end{rem}

\subsection{Limits of  extreme sets of a given cardinality}
\label{subsecInfiniteNumOfSets}
This section is about limits of   extreme sets. The term ``limit'' needs clarification:
Let a sequence of sets $E_j\subset\T$ and $E\subset\T$ be given. We say $E_j\to E$
if
\[
\max\prnb{\sup_{x\in E_j}\,\inf_{y\in E}|x-y|,\ \sup_{y\in E}\,\inf_{x\in E_j}|x-y|}.
\to0
\] 
This is, of course, the Hausdorff distance.\footnote{Alternatively, let  $\mu_j$ be counting measure on $E_j$ for each $j$.
If $\mu_j\to\mu$ weak* in $M(\T)$ and $E=\Supp \mu$
, then
$E$ is the of the $E_j$.}

A finite set in the group \Zp m{} can be thought as a subset of $\Z \mod m$ or it can be
identified with the subset of $\{e^{2\pi i k/m}:0\le k<m\}\subset \T$, where $\T \,(= \R \mod 2\pi)$
is the circle group. When we speak of ``limits'',
we are thinking of the latter representation. However to avoid the clutter (and eyestrain) of
many exponentials, we shall often write ``$k$'' where we mean ``$e^{2\pi i k/m}$'' and
$m$ is implicit.

\begin{lem}\label{lemLimExtremeSet}
Suppose $N>5$ and that there exist primes $2\le p_1<p_2<\dots$
such that
each
\Zp{p_k}{} contains a set $E_k$
 of cardinality $N$ that is the support of an extreme measure.
Then there exists an extreme set $F\subset \mathbb T$ of cardinality $N$, 
a subsequence $p_{k_\ell}$ and  sets 
$F_\ell\subset\Zp{p_{k_\ell}}{}
$ of cardinality $N$ such that  $F$ is the limit of the $F_\ell$, $\#F=N$ and 
$F$ is the support of an extreme measure.
\end{lem}

\begin{proof}
We recall  that if $p_\ell$ is prime, then  \Zp{p_\ell}{} has $p_\ell$ automorphisms.
By renumbering as we go and translating if needed, we may assume  $k_\ell = \ell$ for all $\ell$, as well as
\begin{equation} \label{eqnSizepell}
0\in E_\ell \text{ and }
p_1> 300\, N^2.
\end{equation}
\mpar{deleted the repititious ``We may assume $0\in E$.''. \\
Added a comma after ``$N$''\\
2019-08-08}

We now count automorphisms. Fix an $\ell$. 
Write the elements of $E_\ell$ as $\{e^{2\pi g_{\ell,n}/p_\ell}:0\le n< N\}$. 
For $1\le m< n\le N$,
let $H_{m,n}$ be the set of automorphisms $T$ such that 
\begin{equation}\label{eqnSizepe1}
|Tg_{\ell,m}-Tg_{\ell,n}|\le \frac{2\pi}{10N^2}.
\end{equation} 
The reader will note that here we are using both the representation of \Zp{p_\ell}{} as
$\Z \mod p_\ell$ (using the $g_\ell$ and $T$) and as a subgroup 
of $\T$ (to calculate the distance
between points as  in \eqref{eqnSizepe1}). 

Since there 
\mpar{Inserted ``there''\\2019-08-08}
are  at most $\frac{1+2p_\ell}{10 N^2}$ elements of $G$ that close 
\mpar{Deleted the ``3+'' and added ``+1'', both here and \textit{ff}.\\2019-08-12}
(half on one ``side'', half on the other ``side'' and one in the middle),
there are at most that many automorphisms that carry the two points that close to each other. 
Hence $\bigcup_{1\le m<n\le N}H_{m,n}$ has cardinality at most
 $\binom{N}{2}\frac{3 p_\ell}{10N^2}$, so there are at least\footnote{\, We use
 \eqref{eqnSizepell} and replaced ``$1+2p_\ell$'' with ``$2p_\ell$''. }
 \begin{equation}\label{eqnRemainAutosI}
 p_\ell-\frac{3p_\ell}{10}\ge \frac{7p_\ell}{10}\ge 210\,N^2
\end{equation}
 automorphisms which keep the elements of $E_\ell$ separated by at least $\frac{2\pi}{10N^2}$.

Therefore, for all sufficiently large $\ell$, there exists an automorphism $S_\ell$ of 
\Zp{p_\ell}{} 
such that the elements of $F_\ell = S_\ell E_\ell$ are all
at least $\frac{2\pi}{10N^2}$ apart from each other.

Now for the limits. For each $\ell$ put an extreme measure $\mu_\ell$ on $F_\ell$
(now $F_\ell\subset \T$)  with $\|\mu_\ell\|=N$
and $\|\widehat{\mu_\ell}\|_\infty =\sqrt N$.
We can find a subsequence $\mu_{\ell_k}$ which converges weak-* 
in $M(\mathbb T)$ to a measure $\mu$. Since $|\mu _\ell| \equiv 1$
on $F_\ell$ and
because the elements in the supports of the $\mu_{\ell_k}$ stay apart, the support of $\mu$ has  $N$
elements, $\|\mu\|=N$ and $\mu$ is extreme. The support $F$ of $\mu$ is our limit set and is also thus extreme.
\end{proof}

\begin{rem} There is no guarantee that the limit set is contained in a group of prime order, and, in fact, the next result arranges for the limit set \emph{not}
to be contained in a group of finite order.

\remref{remTCAVNotExtr} shows that the limit argument above can fail
 for TCAV measures since their masses are not necesssarily bounded away from 0.
 \end{rem}
 
\begin{lem}\label{lemLimExtremeSetII}  Let $N>1$. Then there does not exist an 
infinite number of primes $p_k$ such that  \Zp {p_k}{} contains
an extreme set
of cardinality $N$.
\end{lem}

\begin{proof} This is a continuation of the proof of \lemref{lemLimExtremeSet} and we retain the notation of that proof, the assumptions \eqref{eqnSizepell}-\eqref{eqnSizepe1} and 
the properties of the limit set given by \lemref{lemLimExtremeSet}.
In particular we see  from \eqref{eqnRemainAutosI} that there are $\frac{7}{10}p_\ell$ autormorphisms that
keep the points of $E_\ell$ separated for each $\ell$.

Because of \cite[3.1-3.3 and 3.4 (ii)]{MR627683}, we may assume $N>5$.
Since in a group of prime order, all doubletons are equivalent, we
 may assume that for each $\ell$, $g_{\ell,1} = 0$ and $g_{\ell,2} = 1$.

Enumerate \emph{all} the primes in increasing order as $2=q_1<q_2<\cdots$ and enumerate
the elements of the sets $E_\ell=\{0, g_{\ell,2}, \dots, g_{\ell, N}\}$.

For $K=1,\dots,$ let $H_K$ be the cyclic subgroup of $\mathbb T$ of cardinality $M_K =\prod_{k=1}^K q_k^K$.
 By passing to a subsequence of the $p_\ell$, we may assume
that $p_\ell > M_\ell$ for all $\ell$.

For each $\delta>0$ and $2\le n \le  N$,  
 the set of automorphisms $S$ of \Zp{p_\ell}{} 
such that $S(g_{\ell, n})$ is within $\delta$ of an element of $H_K\backslash \{0\}$
 has size at most $2M_K\delta p_{\ell}$.

Let $\delta_K = \frac{1}{2M_K (N-1) 10^K}$. Then there are
$\frac{p_\ell}{10^K}$
automorphisms that carry a non-zero  element of $E$   to within 
$\delta$ of an element of \Zp{M_K}{}.
Thus, there are at most 
\begin{equation}\label{eqCountAutos}
\sum_1^\ell \frac{p_\ell}{10^K}<\frac{p_\ell}{9}
\end{equation}
automorphisms that carry a non-zero element of $E_\ell$ to within $\delta_K$ of \Zp{M_K}{}
for $1\le K<\ell$.
Using \eqref{eqnRemainAutosI}, we see there are at least 
\begin{equation}\label{eqCountAutos2}
\frac{6p_\ell}{9}
\end{equation}
 automormorphisms that  both separate the
elements of $E_\ell$ from each other by at least $\frac{2\pi}{10N^2}$ and also 
keep the non-zero elements
at least $\delta_K$ away from elements of \Zp{M_K}{} for $1\le K<\ell.$

 For each $1\le\ell<\infty$ chose an automorphism $S_\ell$ of \Zp{p_\ell}{} such
 that the elements of $F_\ell=S_\ell E_\ell$ are separated by $\frac{2\pi}{10N^2}$ and for each $1\le K<\ell$ the non-zero elements of $F_\ell$  are at least $\delta_K$
 away from the elements of \Zp{M_K}{}.
 
 Then the limit set $F\subset\T$ of any convergent subsequence of the
 $F_\ell= S_\ell E_\ell$ contains no element of any  $H_K\backslash \{0\}$
 for each $K=1, 2,\dots$.
Let $F$ be such a limit. 

\lemref{lemLimExtremeSet} shows $F$ is extreme. 
Let $H $ be the group generated by $F$. Then $H=\Z^{r}\times L$ where $L$ is finite and $1<r<\infty$.

  Since $0\in F_\ell$ for all $\ell$ we must have
   $0\in F$ but 
\begin{equation}\label{eqnSizePF}
F\cap  \prnb{\{0\}\times L} \text{ contains only the identity of } H
\end{equation}
because $F\backslash\{0\}$ contains no elements of finite order.

Let $\mu$ be an extreme measure on $F$ with $\|\mu\|=\#E=\#F=N$. 
We may assume $\mu$ has
mass 1 at the identity.
 
Let $P: H\rightarrow\Z^r$ 
the the natural projection of abelian groups and $ P'$ the corresponding
projection of measures  
$M(H)\rightarrow M(\Z^{r}).$ Then $(P'\mu)\widehat{\ }= \widehat\mu_{|L^\perp}$. That is,
$(P'\mu)\widehat{\ }$ has constant absolute value of $A>\sqrt 5$ on $\Z$.\footnote{$|\widehat\mu|$ is constant on $\Z$ since it is the weak-* limit of extreme measures 
on $\Zp{p_\ell}{}\subset\T$. Of course $\Z$ is dense in $\widehat H$.}

Because of \eqref{eqnSizePF}, $(P'\mu)(\{0\}) = \mu(\{0\}=1$  
(here ``0'' is -- abusively -- the identity of
the group in question). Since the transform of $P'\mu$ has absolute value $A>1\ge (P'\mu)(\{0\})$, $P'\mu$
cannot be supported only on the identity of $\Z^r$. But the transform of $P'\mu$ has
constant absolute value. This contradicts 
\corref{corExtremeRm} \itref{itcorExtremeRm1}
and completes the proof.
\end{proof}

\begin{thm}\label{thmLimExtremeSetPrimeProds}  Let $N>1$. Then there does not exist an 
infinite number of groups 
$H_k= \Zp{p_{k,1}}{}\times\cdots\times \Zp{p_{k, m_k}}{} $ (all $p_{k,m}$ being prime and $m_k\ge 1$) such
that   each $H_k$ contains
an extreme set
of cardinality $N$.
\end{thm}

\begin{proof}[Proof of \thmref{thmLimExtremeSetPrimeProds}]
We may assume $N>5$ since $1\le N\le 5$ are taken care of by 
\cite[3.1-3.3 and 3.4(ii)]{MR627683}.

Let $E\subset H_k$ have $N$ elements.
It is clear that at most $L=\binom{N}{2}$ coordinates are enough to distinguish 
the elements of $E.$  We may assume those
coordinates are  the first $L$ ones. 
Projecting $\Zp{p_{k,1}}{}\times\cdots\times\Zp{p_{k, m_k}}{}$
onto $\Zp{ p_{k,1}}{} \times\cdots\times \Zp{p_{L,m_L}}{}$ 
will map any extreme measure on
$E$ to an extreme measure on the image of $E$, since  the two norms
 $\|\mu\|$ and $\|\widehat \mu\|_\infty$ are preserved.
Thus, we may assume that the $H_k$ have at most $L$ factors.

The obvious coordinate-wise 
form  of the proof of  \lemref{lemLimExtremeSetII} gives the required conclusion: a subsequence 
of $ (p_{k,1},\dots, p_{k,L})$ can be chosen so that in each coordinate the
projections of the \Zp{p_{k, \ell}}{} accumulate only at the identity and at elements of infinite order and
stay separated by a fixed amount. 

Thus, in the product, the accumulation points are the identity and elements of infinite order. The proof now concludes as
from \eqref{eqnSizePF}, \textit{mutatis mutandi}.
\end{proof}

\begin{rem}
The previous proof fails for groups of prime power order because
there are insufficient automorphisms.
\end{rem} 
\section{A beastiary with remarks}\label{secBeastiary}
Here  are i) a table of extreme sets previously known from  \cite{MR0458059,MR627683},  ii)  two tables  of new extreme sets (``regular'' and ``irregular'') in cyclic groups,
and  iii) a table of new extreme sets in non-cyclic groups. The tables are accompanied by some remarks.

\begin{table}[b]\begin{center}
\begin{tabular} {| c| c|  c|}
\hline
Group & Set (size) & Extreme measure \\ \hline

$G$& $G$ ($\#G$)   & See \cite{MR0458059}\\ \hline

\Zp{4}{}  & 0\dots 2 .\phantom{0} (3) \phantom{*} &$ \delta_0 + e^{3\pi i/4}\delta_1 +i\delta_2 $\\ \hline

\Zp{5}{} & 0\dots 3.\phantom{0} (4)  \phantom{*} & $ \delta_0+\delta_3+e^{2\pi i/3}(\delta_1+\delta_2)   $ \\ \hline

\Zp{6}{}& 0\dots 4.  (5)& $\delta_{0}  +  \delta_{4} -\delta_{2} +   e^{ 3\pi i  / 2}(\delta_{1} +    \delta_{3})    $
\\ \hline
\Zp{7}{} & 0\dots 2, 4. (4) & $ \delta_{0} - \delta_{1} - \delta_{2} - \delta_{4}  $ \\ \hline

\Zp{8}{}& 0\dots 6. (7)& $ \delta_{0}+\delta_{6}+e^{2\pii/3}(\delta_{1}+ \delta_{2}+ \delta_{4}+ \delta_{5} )
$\\&&$
+e^{4\pii/3}\delta_{3}$
\\ \hline
\Zp{2}{3}& (0, 0, 0), (0, 0, 1), (0, 1, 0), &$ \delta_{(0,0,0)}-\delta_{(0,0,1)}-\delta_{(0,1,0)}$
\\
& (0, 1, 1), (1, 0, 0. (5) & $-\delta_{(0,1,1)}+i\delta_{(1,0,0)}$
\\ \hline
\Zp2{}\Times\Zp4{}& (0,0), (0,1), (1,3), (1,0).  (4) & $\delta_{(0,0)}+i(\delta_{(0,1)}+\delta_{(1,0)})-\delta_{(1,3)}
$\\ \hline

\Zp{12}{}& 0, 1, 2,  & $  \delta_{0} +  \delta_{4}  +   e^{ 3\pi i  / 2}\delta_{2}  +   e^{ 5\pi i  / 4}\delta_{3} 
$\\ & 5, 10. (5) & $  
  +   e^{\pi i / 4}\delta_{7}  $
\\ \hline
\Zp2{}\Times\Zp4{}&(0,0), (0,2), (1,0), & $  \delta_{(0,0)}-\delta_{(0,1)}-\delta_{(0,2)}$
\\&(1,2), (0,1). (5)&$+i(\delta_{(1,0)}+\delta_{(1,2)})  $
\\ \hline


\Zp23&(0,0,0), (1,0,0), (0,1, 0),&$  \delta_{(0,0,0)}-\delta_{(0,0,1)}-\delta_{(0,1,0)}$
\\&  (0,0,1), (1,1,0). (5)& $-\delta_{(0,1,1)}+i\delta_{(1,0,0)}  $
\\ \hline
\end{tabular}
  \vskip.125in
\caption{Extreme sets from \cite{MR0458059,MR627683}}
\label{tableFromCite}
\end{center}
\end{table}

\begin{table}[]\begin{center}
\begin{tabular} {| c| c| c|}
\hline
Group & Set (size) & Extreme measure \\ \hline

\Zp{k}{}& $0\dots(k-2).\  (k-1)$,  & See table above and \cite{MR0458059,MR627683} \\
&for $4\le k <6$ \& $k=8 $ &\\  \hline

\Zp{9}{}  & 0\dots 7.\phantom{0} (8) \phantom{*} &$ \delta_{0}+\delta_{7}
+e^{10\pii/7}(\delta_{1}+\delta_{6})
$\\&&$
+e^{4\pii/7}(\delta_{2}+\delta_{5})
+e^{2\pii/7}(\delta_{3}+\delta_{4})$
\\ \hline

\Zp{10}{} & 0\dots 8.\phantom{0} (9)  \phantom{*} & $\delta_{0} +  \delta_{2}+  \delta_{6}   +  \delta_{8} -i(\delta_{1}  +   \delta_{7})    $\\&&$ +    i (\delta_{3} +\delta_{5})     -\delta_{4}  $ \\ \hline

\Zp{12}{} & 0\dots 10. (11) \phantom{*}  & $\delta_{0} +\delta_{4}
-\delta_{8}
+e^{7\pii/4}(\delta_{1}+\delta_{5}+\delta_{9})
$\\&&$ 

-i\delta_{2}
+e^{5\pii/4}\delta_{3}
+i(\delta_{6}+\delta_{10})+e^{\pii/4}\delta_{7}
$  \\ \hline

\Zp{14}{} & 0\dots 12. (13)  \phantom{*} & $ \delta_{0}+\delta_{12}
$\\&&$
+e^{4\pii/3}(\delta_{1}+ \delta_{3} + \delta_{4} +\delta_{8}+ \delta_{9}+\delta_{11})
$\\&&$
+e^{2\pii/3}(\delta_{2}+ \delta_{5}+ \delta_{6}+ \delta_{7}+ \delta_{10})$
  \\ \hline 

\Zp{17}{}&0\dots 15. (16) \phantom{*}& $\delta_{0} +\delta_{9}  +\delta_{15}+e^{6\pii/5}(\delta_{1}+\delta_{2}+\delta_{13}+\delta_{14})
$\\&&$+e^{8\pii/5}(\delta_{3}+\delta_{5} +\delta_{10}+ \delta_{12} )
$\\&&$
+e^{2\pii/5}(\delta_{4}  
+\delta_{6}+\delta_{7}
+\delta_{8}+\delta_{11}) 
$
\\ \hline
\Zp{18}{} & 0\dots 16. (17)  \phantom{*} & $\delta_{0}+\delta_{4}+\delta_{12}+\delta_{16}
$\\&&$+i\delta_{7}+i\delta_{9}
-\delta_{2}
-\delta_{6}-\delta_{8}
-\delta_{10}-\delta_{14}
$\\&&$
-i(\delta_{1}+\delta_{3}+\delta_{5} +\delta_{11}+\delta_{13}  +\delta_{15})
$\\ \hline

\Zp{20}{}&0\dots 18. \phantom{*}(19)& $\delta_{0} +\delta_{5} +\delta_{6} +\delta_{8}+\delta_{9}+\delta_{10} + \delta_{12}
+\delta_{13}$\\&&$
 +e^{4\pii/3}(\delta_{1}+\delta_{3} +\delta_{4} +\delta_{14}+\delta_{15} +\delta_{17})
 $\\&&$
+e^{2\pii/3}(\delta_{2}+\delta_7+\delta_{11}+\delta_{16}+\delta_{18}) $
\\ \hline
\end{tabular}
  \vskip.125in
\caption{``Regular'' extreme sets in cyclic groups}
\label{tableRegular}
\end{center}
\end{table}

\begin{table}[]
\begin{center}
\begin{tabular} {| c| c| c|}
\hline
Group & Set (size) & Extreme measure 
\\ \hline
%
\Zp{10}{}  & 0\dots 4,7 (6) * \dag & $\delta_{0}  +   e^{ 5\pi i  / 6}\delta_{1}  +   e^{ 2\pi i  / 3}\delta_{2}  +   e^{ 5\pi i  / 6}\delta_{3} $\\&&$  +   e^{ 2\pi i  / 1}\delta_{4}  +   e^{\pi i / 6}\delta_{7}$ 
\\ \hline

\Zp{14}{}&0,1,2,3,4,7 (6) \dag & $\delta_{0}  +   e^{ 10\pi i  / 12}\delta_{1}  +   e^{ 8\pi i  / 12}\delta_{2}    $\\&&$   +   e^{ 2\pi i  / 12}\delta_{7}$
\\ \hline

 \Zp{12}{} &0\dots 2,5,6,8,9 (7) \dag & $\delta_{0}  +   e^{ 7\pi i  / 12}\delta_{1}  +  \delta_{3}  +   e^{ 5\pi i  / 6}\delta_{4}  $\\&&$ +    i \delta_{6}  +   e^{ 7\pi i  / 12}\delta_{7}  +    i \delta_{9}  +   e^{ 5\pi i  / 6}\delta_{10}$
\\ \hline

\Zp{16}{}&0\dots 2,4,5,7,11 (7) & $  \delta_{0}     -\delta_{1}  +   e^{ 4\pi i  / 3}\delta_{2}  +  \delta_{4}  +   e^{ 5\pi i  / 3}\delta_{5}  $\\&&$ +   e^{ 5\pi i  / 3}\delta_{7}  +    -\delta_{11} $\\ \hline

\Zp{19}{}&{0\dots 2,5,12,13,15} (7) & $\delta_{0}  +   e^{ 4\pi i  / 3}\delta_{1}  +  \delta_{2}  +  \delta_{5}  +   e^{ 4\pi i  / 3}\delta_{12} $\\&&$   +   e^{ 4\pi i  / 3}\delta_{13}  +   e^{\pi i / 3}\delta_{15}$
 \\ \hline 

\Zp{12}{}	& 0, 1, 3, 4, 6, 7, 9, 10	(8) \dag  &  $\delta_{0}  +   e^{ 7\pi i  / 6}\delta_{1}  +  \delta_{3}  +   e^{ 5\pi i  / 3}\delta_{4}  +   e^{\pi i }\delta_{6} $\\&&$   +   e^{ 7\pi i  / 6}\delta_{7} +   e^{\pi i }\delta_{9}  +   e^{ 5\pi i  / 3}\delta_{10}$
\\ \hline

\Zp{16}{}&0, 1, 4, 5, 8, 9, 12, 13 (8) \dag & $\delta_{0} + e^{ 22 \pi i / 12}\delta_{1} + e^{ 18 \pi i / 12}\delta_{4} $\\&&$+ e^{ 22 \pi i / 12}\delta_{5}  +\delta_{8} + e^{ 10 \pi i / 12}\delta_{9} $\\&&$ + e^{ 18 \pi i / 12}\delta_{12}  + e^{ 10 \pi i / 12}\delta_{13} 
$
 \\ \hline 
\Zp{12}{} & 0\dots 8 (9) \dag & $\delta_{0}  +   e^{ 23\pi i  / 12}\delta_{1}  +   e^{ 3\pi i  / 2}\delta_{2} 
$\\&&$ 
 +   e^{ 17\pi i  / 12}\delta_{3}  +  \delta_{4}   +   e^{\pi i / 4}\delta_{5} $\\&&$  +   e^{ 7\pi i  / 6}\delta_{6}  +   e^{ 5\pi i  / 12}\delta_{7}  +   e^{ 4\pi i  / 3}\delta_{8}$
\\ \hline

\Zp{12}{} &0,1,2,4,5,6,8,9,10 (9) \dag &$
\delta_{0}  +   e^{ 10\pi i  / 6}\delta_{1}  +   e^{ 10\pi i  / 6}\delta_{2}  +   e^{ 8\pi i  / 6}\delta_{4} 
$\\&&$   +   e^{ 102\pi i  / 6}\delta_{5} +   e^{ 2\pi i  / 6}\delta_{6}  $\\&&$ +   e^{ 4\pi i  / 6}\delta_{8}  +   e^{ 10\pi i  / 6}\delta_{9}  +   e^{ 6\pi i  / 6}\delta_{10}$\\ \hline

\Zp{13}{} &0\dots 5, 7, 9, 10 (9) & $\delta_{0} +\delta_{1} - \delta_{2} - \delta_{3} +\delta_{4} - \delta_{5} - \delta_{7}$\\&&$ - \delta_{9} - \delta_{10} $
\\ \hline
\end{tabular}
  \vskip.125in
\caption{``Irregular'' extreme sets in cyclic groups by set size}
\label{tableexceptional}
\end{center}
\end{table}


 \ 

\begin{table}[t]
\begin{center}
\begin{tabular} {| c| c| c|}
\hline
 Group & Set size\  & Extreme measure
\\ \hline
\Zp{2}{3} & 6&$\delta_{(0,0,0)}-i\delta_{(0,0,1)}+e^{7\pii/4}\delta_{(0,1,0)}
$\\&&$
+e^{3\pii/4}\delta_{(0,1,1)}+e^{\pii/4}\delta_{(1,0,0)}+e^{\pii/4}\delta_{(1,0,1)}$ \\ \hline

\Zp42& 6 \dag &   $\delta_{(0,0)}+e^{7\pii/4}\delta_{(0,1)}+i\delta_{(0,2)}+e^{3\pii/4}\delta_{(0,3)}
$\\&&$
+e^{3\pii/4}\delta_{(1,0)}+e^{3\pii/4}\delta_{(1,2)} $  \\ \hline

\Zp42& 6 & $ \delta_{(0,0)}+e^{23\pii/12}\delta_{(0,1)}-i\delta_{(0,2)}+e^{\pii/3}\delta_{(1,0)}
$\\&&$
+e^{11\pii/12}\delta_{(1,1)}+e^{11\pii/6}\delta_{(1,2)} $ \\ \hline

\Zp2{}\Times\Zp4{}& 6 &  $\delta_{(0,0)}+e^{7\pii/4}\delta_{(0,1)}+i\delta_{(0,2)}+e^{3\pii/4}\delta_{(0,3)}
+e^{3\pii/4}\delta_{(1,0)}
$\\&&$
+e^{3\pii/4}\delta_{(1,2)}$   \\ \hline

\Zp2{}\Times\Zp4{}&  6 & $  \delta_{(0,0)}+e^{23\pii/12}\delta_{(0,1)}-i\delta_{(0,2)}
+e^{\pii/3}\delta_{(1,0)}
$\\&&$
+e^{11\pii/12}\delta_{(1,1)}+e^{11\pii/6}\delta_{(1,2)} $\\ \hline

\Zp3{2} & 7  &  $ \delta_{(0,0)}+e^{5\pii/3}\delta_{(0,1)}+e^{2\pii/3}\delta_{(0,2)}+e^{5\pii/3}\delta_{(1,0)}
$\\&&$
+\delta_{(1,1)}+e^{2\pii/3}\delta_{(2,0)}+e^{\pii/3}\delta_{(2,2)}$ \\ \hline

\Zp2{2}\Times\Zp3{}& 8 &   
$\delta_{(0,0,0)}+e^{7\pii/15}\delta_{(0,0,1)}-i\delta_{(0,1,0)}
+e^{29\pii/30}\delta_{(0,1,1)}
$\\&&$
+i\delta_{(1,0,0)}
+e^{29\pii/30}\delta_{(1,0,1)}-\delta_{(1,1,0)}+e^{7\pii/15}\delta_{(1,1,1)}$
    \\ \hline
\Zp4{}\Times\Zp{3}{}& 8 &$\delta_{(0,0)}-i\delta_{(1,0)}+\delta_{(2,0)}-i\delta_{(3,0)}+e^{11\pii/6}\delta_{(0,1)}
$\\&&$
+e^{11\pii/6}\delta_{(1,1)}+e^{5\pii/6}\delta_{(2,1)}+e^{5\pii/6}\delta_{(3,1)}$
\\ \hline
\end{tabular}
  \vskip.125in
\caption{Extreme sets in some non-cyclic groups}
\label{tablenoncyclic}
\end{center}
\end{table}

\begin{rems}\label{remsReTables}
\begin{enumerate}

A dagger (\dag) indicates that the set is discussed in Section \secref{secProofs}.

\item 
In Table \ref{tableFromCite} is, among other things, is an example of an extreme 
 four-element set in \Zp7{}. Seven being prime, that four-element 
 set is neither a coset nor a subset of a five-element coset. All the 4-element
 extreme subsets of \Zp7{} are equivalent.
 We have not investigated the situation in  other groups of prime order, but one could make \conjref{conjexceptionalZprime}.


\item \label{itSixInSeven} Computer calculations show that the six-element set in \Zp7{} is
not extremal and hint that the ten-element set in \Zp{11}{} and the 14-element set in \Zp{15}{} are
also 
\emph{not}  extremal. See  \secref{secComputer} for more on the computer programs and their limitations. 

\item We have found some cyclic 
groups contain two non-equivalent extreme sets of the same cardinality. We provide proofs
for that assertion below. 

\item Computer calculation with the 12-element set in \Zp{13}{} suggest that it is 
extreme.  
 However, the masses suggested by that
 computer calculation do not give a clear indication of what might be an extreme measure,
 those masses involving exponents whose denominator is $2^{22}$. It is possible that replacing
  one or more of the factors of 2 with 11 would produce a better result, but using 
 such a large factor would slow the program down impossibly. Hence, this set does not appear in 
 Table \ref{tableexceptional}.

\item\label{remRealExtremtransform}
 Consideration of the 17-element set in \Zp{18}{} was occasioned 
by the thought that the convolution of
$\mu*\tilde\mu$, when $\mu$ is extreme on the 17-element set, would produce 16 terms at each of
the elements $1,\dots,17$ and that therefore we had the possibility of cancellation if the masses involved were $\pm 1$, $\pm i$ only, which would make for the most rapid machine computation, and that occurred. Something similar might hold for the 65-element set in \Zp{66}{} but the calculation here would likely take $2^{49}$ times as long.

We note that if $\mu$ above is multiplied by $\delta_{-8}$, the resulting measure is
self-adjoint (that is, $(\delta_{-8}*\mu)\widetilde{\ \ } =\delta_{-8}*\mu$). Hence, the set 
$\{-8,\dots,8\}$ is the support of an extreme measure with real transform. 

A similar thought worked for  5th (resp. 3rd) roots of unity in the case of 
16 elements in \Zp{17}{} (resp. 19 elements in \Zp{20}{}).

In general, if $k$ has a small factor, $j$, then there is the  possibility that $j$th roots will appear
as masses of an extreme measure on the set of $k+1$ elements 
in \Zp{k+2}{} and be quick for the computer to find if $j$ is sufficiently small.
  That won't work for 12 elements
in \Zp{13} nor for 18 elements in \Zp{19}{}.

\item If the extreme measure given for $\{0-16\}\subset\Zp{18}{}$ is multiplied by $\delta_{-8}$, the resulting measure is
self-adjoint (that is, $(\delta_{-8}*\mu)\widetilde{\ \ } =\delta_{-8}*\mu$). Hence, the set 
$\{-8,\dots,8\}$ is the support of an extreme measure with real transform. 

\enumisave
\end{enumerate}

\medbreak

That  suggests \conjref{conjnlessone} below.

\medbreak

\begin{enumerate}\enumiget

\item We have found no other six-element extreme sets (other than cosets and homomorphic images of $\{0,1,2,3,4,7\}$) in groups whose order is a multiple of $10$) in any cyclic group of order at most 29. (That search took nearly 8 days of continuous computation.)

\enumisave
\end{enumerate}

\medbreak
 That suggests \conjref{conj6elts}.
 
\medbreak

\begin{enumerate}\enumiget

\item Sets of the form $\{0,g\}+\{0, n\}$ in \Zp{2n}{} and $1\le g<n$ are extremal according to 
 \cite[Theorem 2.1]{MR627683} and are also omitted from the Table \ref{tableexceptional},
 as are  images of 4-element sets in cosets of size 5 and 7.
We also omitted sets contained in subgroups and sets of cardinality a perfect square that
are obtained by application of  \thmref{thmGR2.1} as well as
 those extreme sets that are the sum of an extreme set with a coset. For example,
$\{0,1,2,3,5,6,7,8\}\subset \Zp{10}{}$ is the sum $\{0,5\}+\{0,2,6,8\}$. 
Similarly, 
$\{0,3,6,12\}+\{0,5,10\}\subset \Zp{15}{}$.  

Those sort of examples will complicate
 the process of finding all extreme sets of composite cardinality.
 
 \item Further complicating matters is that the program gives 
 \[ \delta_0+e^{2\pi i/3}(\delta_1+\delta_2)+\delta_3\] a
 s an extreme measure on $\{0,1,2,3\}\subset\Zp5{} $ if one starts  the search
 with third roots of unity, but if one starts with 15th roots of unity, $\delta_{0}+e^{2\pii/15}\delta_{1}+e^{14\pii/15}\delta_{2}+e^{2\pii/5}\delta_{3}
$ is produced.

Similarly, starting the search for 11 elements in \Zp{12}{} with mesh 5 gives an extreme measure with
10th roots of unity (in contrast to the one given in 
Table \ref{tableRegular}, which had 8th roots): 
$ \delta_{0}+\delta_{5}+\delta_{10}
+e^{4\pii/5}(\delta_{1}+\delta_{9})
+e^{2\pii/5}(\delta_{2} +\delta_{8}) 
+e^{8\pii/5}(\delta_{3}+\delta_{4}
)+e^{8\pii/5}(\delta_{6}+\delta_{7})$.

Also, different variants of the search program can give different extreme measures: a late
variant of the search program gave the
extreme measure in Table \ref{tableRegular} for the 13-element set in \Zp{14}{}, $\delta_{0} +   \delta_{2} +     \delta_{10}+  \delta_{12} 
+   e^{ 7\pi i  / 4}(\delta_{1}  +   \delta_{9} )
  +   e^{ 3\pi i  / 4}(\delta_{3}  +   \delta_{8} +   \delta_{11} )
  +   e^{ 3\pi i  / 2}\delta_{4}  +   e^{ 5\pi i  / 4}\delta_{5} 
+   i\delta_{6}  +  e^{ \pi i  / 4}\delta_{7}$
was given by an earlier variant.

\item The 8-element set in \Zp{12}{} is \emph{not} the sum of a coset and an extreme set; see \lemref{lemsumsinZtwlv}
and \propref{prop8in12}.

\item As nearly as we and our (probably poor) programs can tell, the
two 6-element sets in $\Zp2{} \times \Zp{4}{}$ given in Table \ref{tablenoncyclic} are not equivalent.
\end{enumerate}
\end{rems}

\conjref{conjnlessone} does not ``explain'' the extremality of $\{0,1,2,4\}$ in \Zp{7}{}, the extremality of $\{0,1,2,3,4,7\}$ in \Zp{10}{}, nor the extremality of $\{0,1,2,5,6,8,9\}$ in \Zp{12}{}. The computer has looked at subsets of \Zp{11}{} up to size 8 and not found any extreme ones.

\section{Conjectures, questions, and a few proofs of extremality}\label{secConjectures}

\subsection{Conjectures and questions}\label{subsecconjs}
All of these conjectures and questions are suggested by the examples so far or by
numerical evidence -- or by the desire to know how long a computer run will take.

\begin{conj}\label{conjnlessone}
An $n-1$ element subset of \Zp{n}{} is extreme iff $n$ is not congruent to 3$\mod$ 4.
\end{conj}

\begin{conj} \label{conj6elts}
The only six-element extreme sets in cyclic groups are  \Zp{6}{},  $\{0,1,2,3,4,7\}$ $\subset\Zp{10}{}$,
and their images under
group homomorphism and translation. 
\end{conj}
 
\begin{conj} \label{conjexceptionalZprime}
If $p$ is prime and $1<n<p$, all the extreme $n$-element subsets of \Zp p{} are equivalent.
\end{conj}

\begin{conj}\label{conjDiffPS}
If two subsets of \Zp{k}{} have the same cardinality, the same Sidon constant and are in different equivalence classes,
then they are extreme.
\end{conj}

\begin{conj}\label{conjprimecards}
Extreme sets with prime cardinality appear ``more'' frequently than sets whose cardinality is not prime
and not a perfect square.
\end{conj}

 \begin{ques}
 What is the rate of growth of the number of distinct extreme sets of cardinality $N$ in terms of $N$.
 ``Distinct'' means ``not carried to one another by group injections, automorphisms or translations. ''
 \end{ques}
 
 \begin{ques}
 Is there an explicit formula for the number of equivalance classes of size $k$ in a group of order $n$? 
 Or an order of growth in terms of $k$ and $n$?
 \end{ques}
 
 \begin{ques}\label{quesRealTransf}
 What extreme sets support extreme measures with real transforms?
 \end{ques}
\subsection{Extremal sets and their measures for cyclic groups} \label{secProofs}
In this subsection we give a few sample  proofs of some of the extremalities claimed earlier and of related results.
In almost all cases of extremality, we 
leave it to the reader to show that the measure in the relevant table is indeed extreme.
 We also show that the extreme sets here are neither sums of other extreme sets nor of a subgroup and a set as in
 \thmref{thmGR2.1}.

We begin with proving a few extremality results from \cite{MR627683} to illustrate what was omitted from that paper.

\begin{prop}\label{prop3inZ3}
$\Zp{3}{}$ and  $\mu =  \delta_{0}+e^{4\pii/3}\delta_{1}+\delta_{2}
$ are extreme. \newline \end{prop}

{
\begin{proof} 
$\mu * \tilde \mu 
 = \big(1+1+1\big)\delta_{0}
+\big(1+e^{4\pii/3}+e^{2\pii/3}\big)\delta_{1}
+\big(e^{2\pii/3}+e^{4\pii/3}+1\big)\delta_{2}
=3\delta_{0}
$
\end{proof}
}

\begin{prop}\cite{MR627683}\label{prop3in4}
 $\{0,1,2\}\subset \Zp4{}$ and $\nu = \delta(0) + e^{3\pi i/4}\delta(1) +i\delta(2)$ are extreme.
\end{prop}
\omitproof{
\begin{proof}
Then
\begin{align*}
\widehat\nu(0) &= 1- \frac{\sqrt2}2 + \frac{\sqrt2}2i +i, \text{ so }
\\  
|\widehat\nu(0)| &= \big|(1-\frac{\sqrt2}2)^2 +(1+\frac{\sqrt2}2)^2\big|^{1/2} =\sqrt3.
\\
\widehat\nu(1) &= 1 -\frac{\sqrt2}2 -\frac{\sqrt2}2i -i, \text{ so }
|\widehat\nu(1)|=  \sqrt3\\
\widehat\nu(2) & = 1 +\frac{\sqrt2}2 -\frac{\sqrt2}2i +i \text{, so } |\widehat\nu(2)|= \sqrt3.
\\
\widehat\nu(3) & = 1  +\frac{\sqrt2}2 +\frac{\sqrt2}2i -i, \text{ so } 
|\widehat\nu(3)| = \sqrt3.\qedhere
\end{align*}
\end{proof}
}
Another extreme measure is $\mu=\delta(0) - e^{3\pi i/4}\delta(1) + i\delta(2)
=\delta(0) + e^{7\pi i/4}\delta(1) + i\delta(2)$.
All extremal measures have one of the forms, $\nu$ or $\mu$:
Suppose $\mu=\delta_0+a\delta_1+b\delta_2$ is extreme on $\{0,1,2\}\subset \Zp4{}$. Then
$\mu*\tilde\mu=3\delta_0+(a+\bar ab)\delta_1 + (b+\bar b)\delta_2 +
(\bar a+a\bar b)\delta_3=3\delta_0$.
Hence, $b+\bar b=0$ so $b=\pm i$. Assume $b=i$. Then 
$\bar a-ia=0$ so\footnote{\ 
Let $a= x+iy$. Then $\bar{a}-ia= x-iy-ix+y=0$ means $x+y=0$. Since 
$|a|=1$, $x= \pm\sqrt2/2$.}
either $a=\pm e^{\pm\pi i/4}$ or $a =\pm e^{\pm 3\pi i/4}$. 
Since $y=-x$, we have $a=\pm \exp(3\pi /4).$

\medskip
Every three-element extreme set can be obtained from \Zp{3}{} and a 3-three element subset of \Zp{4}{}
by the operations of group automorphism, passing to a subgroup, and translation
\cite[3.1(ii)]{MR627683}. See \cite{ABeastiaryAppendix} for a proof of part of \cite[3.1(ii)]{MR627683}. The complications of the proof there illustrate
 why the proofs
of \cite[3.1-3.3]{MR627683} occupied 200 pages of manuscript.

\subsection{Non-equivalent extreme sets}

\begin{prop}
\cite{MR627683}
\label{prop5in12} The subsets $E=\{0,2,4,6,8\}$, $F=\{0,2,3,4,7\}$ of \Zp{12}{} are extreme but
not equivalent.
\end{prop}
\begin{proof}
Non-equivalence: 
the group automorphisms of \Zp{12}{} are multiplication by 5, 7 and 11 (all mod 12). Each of them takes odd elements of \Zp{12}{} to odd elements and even elements to even elements. Translations either take evens to evens and odds to odds or evens to odds and odds to evens. Thus, no combination of group automorphisms and translations can take 
$E$, whose image will contain either only evens or only odds, onto $F$, which contains both evens and odds.

$\{0,2,4,6,8\}$ is extreme because it is a five-element subset of the the coset  $\{0,2,4,6,8,10\}$ in \Zp{12}{}.

The proof of the extremality of the second set is in \cite{ABeastiaryAppendix}.
\end{proof}

\begin{prop}\label{prop7in16} Each of 7-element sets
\begin{enumerate}
\item
 $\{0,2,4,6,8,10,12\}$ and 
 \item $\{0,1,2,4,5,7,11\}$  
 \end{enumerate}
 is extreme in \Zp{16}{}.
They are not equivalent.
\end{prop}

\begin{proof}
The sets are  not equivalent because the automorphisms 
(multiplication by odd integers) preserve the parity  of
 elements and translation switches or leaves fixed the 
 parity. Since the first set has both odd and even elements
 so will every set equivalent to it. (We will use this argument
 again in \propref{prop8in16}.)
 
 As for extremality, first note that $\{0,2,4,6,8,10,12\}$ is extreme because it is a 7 element subset of an 8 element coset
and so extreme. For the other set, use the measure in Table \ref{tableexceptional}. Further details are
in \cite{ABeastiaryAppendix}.
\end{proof}

\begin{prop}\label{prop8in12} The 8-element set
$\{0,1,3,4,6,7,9,10\}$ 
 is extreme in \Zp{12}{} and is not the sum of two extreme sets in spite of being the sum $ \{0,1,3,4\} +\{0,6\}$
 and also the sum $\{0,3,6,9\}+\{0,1\}$.
\end{prop}

\begin{proof}
The reader can verify for herself that the measure given in Table \ref{tableexceptional} for this set is extreme.

In the proof of non-equivalency we apply \lemref{lemsumsinZtwlv} below and the fact that $\{0,1\}$ is not extreme in \Zp{12}{}\footnote{The difference 
$\{0,1\}-\{0,1\}$ produces the terms $0-0$, 1-1, 0-1, 1-0 and so $\{0,1\} \subset \Zp{12}{} $ fails the test of \corref{corExtrCnxln}.} 
\end{proof}

\begin{lem}\label{lemsumsinZtwlv}
\Zp{12}{} has no 8-element set that is a sum of a four-element extreme set and a coset.
\end{lem}

\begin{proof}
Suppose $E= A+B$ has 8 elements, where $A$ has 4 elements and $B$ has two. We may assume $0\in A$ and $B=\{0,6\}$.
Clearly $A$ cannot contain 6. \corref{corExtrCnxln} and calculation (both by hand and machine) show that if a 4-element set in \Zp{12}{} 
 lacks 6, it is not 
extreme.\footnote{\,
The only 4-element extreme sets of \Zp{12}{} (up to equivalence) 
are $\{0,1,6,7\},$ $\{0, 2,6,8\}$ and $\{0,3,6,9\}$, all containing 6. Several other four-element sets pass the test of
\corref{corExtrCnxln} but are not extreme.}
\end{proof}



\begin{prop}\label{prop8in16} The 8-element set
$\{0, 1, 4, 5, 8, 9, 12, 13\}$ 
is  extreme in \Zp{16}{} even though it is not equivalent to a coset of $\{0,2,4,6, 8,10,12,14\}$.
\end{prop}

\begin{proof}
The set is   not equivalent to a coset because the autormorphisms 
(multiplication by odd integers) preserve the parity  of
 elements and translation switches or leaves fixed the 
 parity. Since this set has both odd and even elements, 
 so will every set equivalent to it. 
 
 We leave to the reader that the measure given in Table \ref{tableexceptional} is extreme.
\end{proof}

\begin{prop}\label{prop9in12} The 9-element sets
$\{0,1,2,3,4,5,6,7,8\}$ and \\$\{0,1,2,4,5,6,8,9,10\}$ are extreme in \Zp{12}{}. They are  non-equivalent. 
\end{prop}

\begin{proof}
There are two ways to prove non-equivalence. First, by  a tedious calculation (which we delegated to a computer). The second is to show that one of the sets is the sum
of a coset with a three-element set and the other is not, which we do in the next paragraph.

The only 3-element subgroup in \Zp{12}{} is $\{0,4,8\}$. If $\{0,4,8\}+\{a\} 
\subset\{0,1,2,3,4,5,6,7,8\}$, then $a\ne\, 1,2,3$,
 since $8+1=9,8+2=10, 8+3=11$ are 
  not in $E.$  Similarly,
   $a\ne 5,6,7,9,10,11$
 since none of those elements is in $E$. Thus, $a\in \{0,4,8\}$ and the first set is seen not be a sum.
 
 Of course, $\{0,4,8\}+\{0,1,2\}=\{0,1,2,4,5,6,8,9,10\}$ and both final conclusions follow.

Proofs that the measures given in Table \ref{tableexceptional} are extreme are given in
\cite{ABeastiaryAppendix}.
\end{proof}

\section{The computer programs}\label{secComputer} 
\subsection{General descriptions}
We have two sets of programs: those  searching for extreme sets and those  checking
that a putative extreme set is  ex\-treme. In each set there are 
separate programs for cyclic
and non-cyc\-lic groups. The search programs for cyclic groups are of two types: those which examine the Fourier-Stieltjes transforms
of measures, and those which look at the convolution $\mu*\widetilde\mu$ of candidate measures to see
if the coefficients of that product are (close to) zero except at the identity.

The search programs divide into two classes of two streams each:
\begin{itemize}
\item findBest -- searchs for measures with minimal FST
\item findX -- searches for extreme measures (and keeps fewer candidates than findBest)
\end{itemize}
In each case
the program saves a list of promising candidates and discards less promising ones, assuming
there are not too many ``promising'' ones. Each saved
measure is used as a starting point for  an increased mesh, usually doubled, at the next pass.\

The programs are not perfect, are in my poor C, and have been run
only in  (Debian 8 \& 9) Linux  BASH terminal windows --
 there is no proper user interface, much less a GUI front end. 
One edits the source code slightly and recompiles for each group and set size(s). Whether my programs
 can be run
 in other operating systems without changes I do not know, though
  for other versions of Linux the answer is almost surely yes.
  
  The programs that look for extreme sets can be assigned to use up to 16 GB of resident RAM for
  storing promising measures for each group/setsize pair. Those programs are easily modified to 
  use less RAM, in which case they 
  will also run faster and give   less reliable results (see next section). 
  
\subsection{Confidence in the results} This depends on the answer a search program gives.
 Each time a program
said, ``this set is extreme and here is an extreme measure,'' the program was correct. If the program
said, ``there's no extreme measure on this set,'' one can be confident \emph{only if} a) the
program has looked at all possible measures directly, or b) if the program shows that
 the measures not considered cannot be extreme because
they are too close to non-extreme ones (a little differential calculus is useful in
setting up the relevant inequalities). For example, in the case of six elements in \Zp7{}, the
search programs showed that  the square of the PSC is at least  6.74670307754671 
with an error of at most 0.000467467. Hence, the set is not extreme. Alternative calculations, one written in $C$ by the author
and the other
in Mathematica by L. T. Ramsey, show that each candidate measure $\mu$ is such that $\mu*\tilde\mu-6\delta_0$
has at least one coefficient that's greater than 0.3.

The programs (whether looking at the FST or looking at $\mu*\widetilde\mu-(setsize)\delta_0$)  discard
candidates if there are too many to keep for further examination. This means  a ``no extreme set here''
result cannot be relied upon. 

\subsubsection{ 10 elements in \Zp{11}{}} The case of 10 elements in \Zp{11}{} is instructive. The
 search for an extremal measure
 reported  discarding  many billions  of candidates. Not discarding candidates would have required
least 43 GB of RAM to get beyond an initial mesh of 8 (that took 1 minute 23 seconds) 
and 19 terabytes of RAM with an initial mesh of 16 (which took 11.5 hours).   Writing candidates to disk and reading them back are even more time-consuming activities in and of themselves.
 
Increasing the starting fineness of the initial search
  (a possible way to reduce the potential number of candidates stored) will increase the time
  needed for the first pass: starting   with mesh = 16  takes about 11.5 hours and starting  at 32
 will take more than 9 months, with no assurance that the number of candidates needed to look at will be fewer
 than for mesh = 16. 
 
 Looking at the mesh 8 case again: there were about 80,000,000 candidates discarded. To search each of them
   to final mesh of 64 I estimate would take a bit more than two centuries.
   
   \subsubsection{More generally, $n-1$ elements in \Zp{n}{}} Each point mass in the expansion of $\nu=\mu*\tilde\mu$ 
   (other than at the  identity) has $n-1$ terms. For them to sum to zero they must involve either the
   $(n-2)^{th}$ roots of unity or $k^{th}$ roots of unity where $k$ divides $n-1$, possibly rotated.
   Thus, the most promising search involves starting with one of those $k$s. Of course, for 12 elements in \Zp{13}{},
   this means starting with an initial mesh of 11, which would take several years to complete.  We started the
   search for an extremal measure on the 23-element subset of \Zp{24}{}. It took 23.5 seconds to complete 
   that initial search. Assuming that 23 of those seconds were in setting up, that means that some 6000
   millenia would be needed (for our computer program on our hardware) to do mesh 4 with no discarding of
   candidates.  Alternatively, to process, say, 10 million candidates saved by the mesh 2 pass for mesh 4
   would take  mere additional 38 years.
   
   Of course, in $n$ is not too large and $n-2$ has a small 
   factor, finding an extremal measure may be possible, as the cases of $n$=12, 14, 17, 18, and 20
   show.
\smallbreak 
 
 \subsection{Can we go further?}
 Not much. As the above indicates,   a ``not extreme'' report for a set with 10 or 
 more elements cannot be believed and furthermore, the time and memory demands increase so rapidly that, absent
 more than several orders of magnitude in computational power, 
 there will be no computationally trustworthy ``not extreme'' results
 for these larger sets in the forseeable future.
 
 ``Is extreme'' results \emph{are} easier to come by as indicated, but of course each additional element in the
 set increases search time by at a factor of 2 for initial mesh = 2,
  and initial meshes greater than 3 are not feasible 
 for larger sets.
 
 \subsubsection{Summary of particular cases}
 If an $n-1$ element subset of \Zp{n}{} ($n\le 15$) does not appear in Table \ref{tableRegular}, then the comuter
 program has given an unreliable ``not extreme'' conclusion.

 \subsection{More on searches}
 
 \subsubsection*{Bounding the PSC} This gives an upper bound of the PSC. 
 If this type of search gives PSC equal to the square root
 of the setsize, we have found an extremal set. If the search fails to show the
 PSC is the root of the set  size, we  cannot conclude  that the
 set is not extremal unless nothing was thrown away and the search is fine enough.
 
 \subsubsection*{Bounding the mass of $\nu*\widetilde\nu-(setsize)\delta_0$} 
 
Assume $\mu$ is extremal, so $\mu*\widetilde\mu = (setsize)\delta_0$. 
Let $\nu$ be a test measure\footnote{\ Both measures here have mass $setsize$ with weights of equal 
absolute value at each support point and unit point mass at the identity.}. Then
  \begin{equation}\label{eqClose1} 
  \|(setsize)\delta_0-\nu*\widetilde\nu \| =\|\mu*\tilde\mu -\nu*\tilde\nu\|\le 2 (setsize) \|\mu-\nu\|.  
  \end{equation}
The above is in turn is bounded on the right by
\begin{equation}\label{eqClose2}
\varepsilon =
2 \,(setsize)\, (setsize -1) \frac{2\pi }{mesh},     
\end{equation}
when $\nu$ is the measure closest to $\mu$ in our search
and $mesh$ is the number of distinct masses considered at each point of $E$ (giving rise to time estimates
of the order of $(mesh)^{setsize-1}$).
If all the measures looked at with a particular $mesh$ satisfy
\begin{equation}\|\nu*\widetilde\nu\| > setsize + \varepsilon,
\end{equation}
then no finer search (i.e., with larger $mesh$'s) will produce an extremal measure . Hence, 
there can be no extremal measure on that set.

This means that we make our lists of measures that satisfy
\[
\|\nu *\widetilde\nu\| \in [setsize-PRECISION, setsize +\varepsilon+PRECISION].
\]
Now, to speed up the program, instead of calculating $|\nu*\widetilde\nu(\{g\})|$,
which would involve two squares and a square root for each of $grouporder-1$ points,
we calculate 
\[
|\Re \prn{\nu*\widetilde\nu(\{g\}}| + |\Im\prn{\nu*\widetilde\nu(\{g\}}|
\]
at $grouporder-1$ points.
 That introduces a factor of
$\sqrt{2}$
and means that we make our lists using measures that satisfy
\[
\|\nu *\widetilde\nu\| \in [setsize-PRECISION, setsize +\varepsilon\sqrt2+PRECISION].
\]
I usually set PRECISION at $10^{-7}$. Changing it by a factor of 10 does not affect results much.

\subsection{\ \ Time complexity} \label{secTimeComplexity}  As discussed above,  I do not have much hope of going beyond sets with 10 or 11 elements in groups of size below 25 (the 17-element set
  in \Zp{18}{} was a lucky guess):  my (uncompleted) searches for 9 element sets
   have taken several weeks using a reasonably fast CPU (Ryzen 7). Doubling or tripling the speed would add 
   perhaps one or two elements to the set size that my program can do in a reasonable time\footnote{\ Defined as the interval between power cuts in Whitehorse of more than 15 minutes, 15 minutes being the
   estimated capacity of my UPS.}, as computation time is
  at best exponential in set size (that is, $O(L\cdot mesh^{setsize)} $, where $L$ is the  number of equivalence classes, which  could be several hundred) plus another bit that's \proofBitFormat{binom(groupsize, set size)}.  the binomial part comes from the search for equivalence classes.
  If the program takes more than 2 days to find equivalence classes for a particular choice of group and setsize, I 
  kill it.
  
The situation is worse for non-cyclic groups because adding a
 factor doubles (or worse) the size of the group and
vastly increases the number of ``equivalence'' classes the program 
generates as well as the time to generate them. 

Here are some examples, given to assuage my conscience about killing searches.
\begin{itemize}
  
\item The search for 7-element sets in $\Zp{2}{2}\times\Zp5{}$ (20 group elements)
 produced  58 ``equivalence'' classes and took about 5 days to complete with a max mesh of 720.

\item Big sets take time. The search for 9-element sets in \Zp42  got into the 
second class of 40 on its 7th day and was then killed. Estimated time to complete is thus 200-280 days.

\item Big groups lead to big problems. The search for 6-element sets in \Zp72 (49 group elements) spent several
 days trying to generate the list of equivalence classes
  before I killed it. I have no idea how many equivalence classes it 
  would have found nor how long to compute all their PSCs.

\end{itemize}

\bibliographystyle{amsplain} 
\bibliography{ccg}

\begin{abstract}
This appendix contains  details not intended for publication in  
the paper  of which
it is an appendix,
but which details might be useful in the public record.
\end{abstract}

\maketitle

\centerline{\today}
\tableofcontents

\end{comment}
\addtocontents{toc} {\protect\vspace{10pt} }
\addcontentsline{toc}{section}{\qquad\qquad{\textsc{Appendix}}}
\addtocontents{toc}{\protect\vspace{2pt}}

\subsection*{Note to the referee and other readers} This appendix contains material that is not 
intended for publication because it has too many details, is too reptitive
 and also duplicative of material in the main paper. It is planned to be 
included with the pdf electronicly distributed by the author 
so that readers can see the results of the computer 
programs in fuller detain than can be justified in a publication, even an electronic one.

\section{Proofs of various items}
Here is the special case of 
\cite[Thm. 4.1]{MR2363058} that we need.

\begin{prop}\label{propAutosOfZphatk} Let $1\le m$ and $p$ a prime. then
$L=\Z_{p^m}$ has $p-1$ automorphisms.
\end{prop}

\begin{proof}
Every element $x\in L$ has the form
\[
x=
a_0+a_1p+\cdots a_{m-1}p^{m-1}
\]
where $0\le a_k<p$ for $0\le k <m$.
Every automorphism $T:L\to L$ is determined by $T(1)$.  Using the facts that $1$ is
a generator of $L$ and $T$ is an automorphism, we see that 
\begin{align}
T(x) &= T\prn{a_0+a_1p+\cdots+a_{m-1}p^{m-1}}
\\&
= T(a_0)+T(a_1)p+\cdots+T(a_{m-1})p^{m-1}.
\end{align}
\end{proof}

\begin{prop}There is no real extreme measure on $\{0,1,2,4\}\subset\Zp 7{}.$
\end{prop}
\begin{proof}
\textit{Look at the 4 possible measures.}
Case I.

$\mu=\delta_0+\delta_1+\delta_2-\delta_4.$
$M=\begin{pmatrix}
1&1&1&0&-1&0&0\\
0&1&1&1&0&-1&0\\
\dots
\end{pmatrix}
$

Row 2 inner row 1 is  $=2\ne0$.

\smallskip

Case II.

$\mu=\delta_0+\delta_1-\delta_2+\delta_4.$
$M=\begin{pmatrix}
1&1&-1&0&1&0&0\\
0&1&1&-1&0&1&0\\
0&0&1&1&-1&0&1\\
\dots
\end{pmatrix}
$

Row 3 inner row 1 is $-2=\ne 0$.

\smallskip

Case III.

$\mu=\delta_0-\delta_1+\delta_2+\delta_4.$
$M=\begin{pmatrix}
1&-1&1&0&1&0&0\\
0&1&-1&+1&0&1&0\\
0&0&1&-1&1&0&1\\
\dots
\end{pmatrix}
$

Row 2 inner row 1 is $=-2\ne 0$.

\smallskip

Case IV.

$\mu=-\delta_0+\delta_1+\delta_2+\delta_4.$

$M=\begin{pmatrix}
-1&1&1&0&1&0&0\\
0&-1&1&1&0&1&0\\
0&0&-1&1&1&0&1\\
1&0&0&-1&1&1&0\\
\dots
\end{pmatrix}
$

Row 3 inner row 1 is $=-1\ne 0$.
\qedhere
\end{proof}

\section{Proofs of extremalities for cyclic groups}
In this section we give   proofs of some of the extremalities claimed in and of related results.
 We also show that the extreme sets here are neither sums of other extreme sets nor of a subgroup and a set.

To save the reader from flipping between \cite{ABeastiary} and this document, we sometimes include details
from \cite{ABeastiary}.

\subsection{Sets with three elements}

\begin{prop}\label{prop3inZ3}
$\Zp{3}{}$ is extreme. \newline \end{prop}

\begin{proof}
Let  $\mu =  \delta_{0}+e^{4\pii/3}\delta_{1}+\delta_{2}
$.  Then 
\begin{align*} 
\mu * \tilde \mu & = \big(1+1+1\big)\delta_{0}
\\&\qquad
+\big(1+e^{4\pii/3}+e^{2\pii/3}\big)\delta_{1}
\\&\qquad
+\big(e^{2\pii/3}+e^{4\pii/3}+1\big)\delta_{2}\\&=3\delta_{0}\qedhere
\end{align*}
\end{proof}

In \Zp{4}{} there is only one, up to translation, and it is extreme:
\begin{prop}\cite{MR627683}\label{prop3in4}
 $\{0,1,2\}\subset \Zp4{}$ is extreme.
\end{prop}
\begin{proof}
Let $\nu = \delta(0) + e^{3\pi i/4}\delta(1) +i\delta(2)$.
Then
\begin{align*}
\widehat\nu(0) &= 1- \frac{\sqrt2}2 + \frac{\sqrt2}2i +i, \text{ so }
\\  
|\widehat\nu(0)| &= \big|(1-\frac{\sqrt2}2)^2 +(1+\frac{\sqrt2}2)^2\big|^{1/2} =\sqrt3.
\\
\widehat\nu(1) &= 1 -\frac{\sqrt2}2 -\frac{\sqrt2}2i -i, \text{ so }
|\widehat\nu(1)|=  \sqrt3\\
\widehat\nu(2) & = 1 +\frac{\sqrt2}2 -\frac{\sqrt2}2i +i \text{, so } |\widehat\nu(2)|= \sqrt3.
\\
\widehat\nu(3) & = 1  +\frac{\sqrt2}2 +\frac{\sqrt2}2i -i, \text{ so } 
|\widehat\nu(3)| = \sqrt3.\qedhere
\end{align*}
\end{proof}

Another extreme measure is $\mu=\delta(0) - e^{3\pi i/4}\delta(1) + i\delta(2)
=\delta(0) + e^{7\pi i/4}\delta(1) + i\delta(2)$.
All extremal measures have one of the forms, $\nu$ or $\mu$:
Suppose $\mu=\delta_0+a\delta_1+b\delta_2$ is extreme on $\{0,1,2\}\subset \Zp4{}$. Then
$\mu*\tilde\mu=3\delta_0+(a+\bar ab)\delta_1 + (b+\bar b)\delta_2 +
(\bar a+a\bar b)\delta_3=3\delta_0$.
Hence, $b+\bar b=0$ so $b=\pm i$. Assume $b=i$. Then 
$\bar a-ia=0$ so\footnote{\ 
Let $a= x+iy$. Then $\bar{a}-ia= x-iy-ix+y=0$ means $x+y=0$. Since 
$|a|=1$, $x= \pm\sqrt2/2$.}
either $a=\pm e^{\pm\pi i/4}$ or $a =\pm e^{\pm 3\pi i/4}$. 
Since $y=-x$, we have $a=\pm \exp(3\pi /4).$

\medskip
Every three-element extreme set can be obtained from \Zp{3}{} and a 3-three element subset of \Zp{4}{}
by the operations of group automorphism, passing to a subgroup, and translation
\cite[3.1(ii)]{MR627683}. Here is a proof of part of \cite[3.1(ii)]{MR627683}. 
The complications of the proof here illustrate
 why the proofs
of \cite[3.1-3.3]{MR627683} occupied 200 pages of manuscript.

\begin{prop}[\cite{MR627683}]\label{prop3eltsets} 
If  $3\le k$ and $E=\{0,a,b\}$ is extreme in $\Zp{k}{}$ then either $E$ is a subgroup or it is a three-element
subset of a four-element subgroup.
\end{prop}

\begin{proof} Suppose $E=\{0, a,b\}\subset \Zp{k}{}$ is extreme. Let $\mu = \delta_0 +\delta_a+\delta_b.$
We may assume $0<a<b<k$. Then 

\begin{equation}\label{eq1prop3eltsets}
\mu*\tilde\mu= 3\delta_0+\delta_a+\delta_b +\delta_{-a}+\delta_{-b}+\delta_{a-b}+
\delta_{b-a}.
\end{equation}
 If $a\ne -b$ and $a\ne -a$,  then we have four pointmasses, $\delta_{\pm a},$ $ \delta_{\pm b}$ which cannot all
be matched by $\delta_{a-b},$ $\delta_{b-a}$ and $E$ is not extreme by \cite[Cor. 2.3]{ABeastiary}. 

\smallbreak
Case I: Suppose $a=-a$, that is, $a=k-a.$ Then $k=2a$ and $b\ne -b$ so $a-b\ne a+b$.
Thus,\[
\mu*\tilde\mu= 3\delta_0+2\delta_a+\delta_b +\delta_{-b}+\delta_{a-b}+
\delta_{b+a}.
\]
If $a-b=b$ then $a=2b$ and $k=4b$ so we have three elements of a  4 element subgroup, $\{0, b, 2b\}\subset\Zp{4b}{}.$
Therefore, we may assume $a-b\ne b$. 

Clearly $a-b$ is distinct from $0$, $a$, $b$, and $a+b$. Hence $\mu$ is not extreme in this case.

\smallbreak
Case II: $b=-b$, so $a\ne-a$. Then 
\[
\mu*\tilde\mu= 3\delta_0+2\delta_b +\delta_a+\delta_{-a}+\delta_{a+b}+
\delta_{b-a}
\]
or
\[
\mu*\tilde\mu= 3\delta_0+2\delta_b +\delta_a+\delta_{-a}+\delta_{a+b}+
\delta_{-b-a}.
\]
If $a=-b-a$, then $b=2a$ and $k=4a$ and we are again in the situation of a 3-element subset of a 4-element subgroup.
Therefore we may assume $a\ne -b-a$, in which case $\delta_{-b-a}$ appears only once in \eqref{eq1prop3eltsets} and $E$ is not extreme.
\smallbreak

Case III: 
$a=-b$ so 
\[ 
\mu*\tilde \mu = 3\delta_0+2\delta_a+2\delta_b + \delta_{a-b}+\delta_{b-a}.
\] 

We have two subcases. First, suppose $a= b-a$. Then since $a=-b$, $a= b--b=2b=-2a$ so $3a=0$. Hence $E=\{0,a,2a\}$ and $k=3a$, that is, $E$ is a subgroup.

Second subcase: $a-b=b-a$. We still assume $a=-b$.
  Since $a\ne b$ we must have $a-b=\ell$ has order 2, and so $k=2\ell.$ Now $a-b = \ell = -2b = 2a$ $\mod\,2\ell$. Hence, $2a=\ell,$
 and $a$ has order 4. Thus, $E=\{0,a, 2a\}\subset\Zp{4a}{}$ and $E$ is a three element subset of a 4-element subgroup. That takes care
 of the third and final case.
\end{proof}

\subsection{Sets with four elements} 

\begin{prop}\label{prop4in4}
$\Zp{4}{}$ is extreme. \newline \end{prop}

\begin{proof}
Let  $\mu =  \delta_{0}+e^{5\pii/3}\delta_{1}+\delta_{2}+e^{2\pii/3}\delta_{3}
$.  Then 
\begin{align*} 
\mu * \tilde \mu & = \big(1+1+1+1\big)\delta_{0}
\\&\qquad
+\big(e^{4\pii/3}+e^{5\pii/3}+e^{\pii/3}+e^{2\pii/3}\big)\delta_{1}
\\&\qquad
+\big(1-1+1-1\big)\delta_{2}
\\&\qquad
+\big(e^{\pii/3}+e^{5\pii/3}+e^{4\pii/3}+e^{2\pii/3}\big)\delta_{3}\\&=4\delta_{0 }.\qedhere
\end{align*}
\end{proof}

\begin{prop}
\cite{MR627683}
\label{prop4in5}
$\{0,1,2,3\}$ is an extreme subset of $\Zp5{}$.
\end{prop}
\begin{proof}
Here is an extremal measure:
$\mu = \delta_0+e^{2\pi i/3}(\delta_1+\delta_2) +\delta_3$.
Then 
$
\tilde\mu = \delta_0 +e^{-2\pi i/3}(\delta_3+\delta_4)+\delta_2$, so 
\[\mu*\tilde\mu = 4\delta_0+(1+e^{2\pi i/3} +e^{-2\pi/3})(\delta_1+\delta_2+\delta_3+\delta_4)\\
= 4\delta_0. \qedhere
\]
\end{proof}

\begin{prop}
\cite{MR627683}
\label{prop4in7}
$\{0, 1, 2, 4\}$ is extreme in \Zp{7}{}.
\end{prop}

\begin{proof}
Let $\mu = \delta_{0} - \delta_{1} - \delta_{2} - \delta_{4} 
$.  Then 
\begin{align*} 
\mu * \tilde \mu & = 4\delta_{0}   +  \big(  -1 + 1\big)\delta_{1} 
 +  \big(  -1 + 1\big)\delta_{2}  +  \big(  -1 + 1\big)\delta_{3} 
 \\&\qquad 
  +  \big(1-1\big)\delta_{4}   +  \big(  -1 + 1\big)\delta_{5} 
   +  \big(  -1 + 1\big)\delta_{6} 
=4\delta_0.\qedhere
\end{align*}
\end{proof}

\begin{proof}[Alternative proof of \propref{prop4in7}]
Let $z =e^{2\pi i/7}$.

We compute the transform of $\mu$, using arithmetic mod 7 in
the exponents of $z$.
\begin{align*}
\widehat\mu(0) &= 1 -1-1-1=-2.\\
|\widehat\mu(1)|^2 &= -z^4 + z^3 + z^{-3} - z^{-4} + 4
\\ \notag
&= -z^4+z^3+z^4-z^3+4=4,
\\
|\widehat\mu(2)|^2 &=
-z^8 + z^6 + z^{-6} - z^{-8} + 4=4,
\\
|\widehat\mu(3)|^2 &=
-z^5-z^2-z^5-z^2+4=4,
\\
|\widehat\mu(4)|^2 &=
%
-z^2+z^5+z^2-z^5+4=4,
\text{ and}
\\
|\widehat\mu(6)|^2 &= 
-z^3+z^4+z^3-z^4+4=4.
\qedhere
\end{align*}
\end{proof}

\bigbreak
\subsection{Sets with five elements}

\begin{prop}\label{prop5in6}
\cite{MR627683}
$\{0,1,2,3,4\}$ is extreme in \Zp{6}{}.
\end{prop}

\begin{proof} Let $\mu =  \delta_{0}  +   e^{ 3\pi i  / 2}\delta_{1}  +    -\delta_{2}  +   e^{ 3\pi i  / 2}\delta_{3}  +  \delta_{4}$.
Then
\begin{align*}
 \mu * \tilde\mu&=  5\delta_{0}   +  \big( -i - i+  i +  i \big)\delta_{1} 
  +  \big(1-1+1-1\big)\delta_{2} 
  \\& \qquad +  \big(  i - i- i+  i \big)\delta_{3} 
   +  \big(  -1 +1-1+ 1\big)\delta_{4} 
  \\& \qquad
   +  \big(  i +  i - i+ e^{ 3 \pi i}\big)\delta_{5}  
   \\&=5\delta_0.\qedhere
 \end{align*}
\end{proof}

\begin{prop}
\cite{MR627683}
\label{prop5in12} The subsets $E=\{0,2,4,6,8\}$, $F=\{0,2,3,4,7\}\subset \Zp{12}{}$
are not equivalent but both are extreme.
\end{prop}

\begin{remark} This does not contradict \cite{MR627683}, since one set is in a subgroup.
\end{remark}

\begin{proof} 
Non-equivalence: 
the group automorphisms of \Zp{12}{} are multiplication by 5, 7 and 11 (all mod 12). Each of them takes odd elements of \Zp{12}{} to odd elements and even elements to even elements. Translations either take evens to evens and odds to odds or evens to odds and odds to evens. Thus, no combination of group automorphisms and translations can take 
$E$, whose image will contain either only evens or only odds, onto $F$, which contains both evens and odds.

$\{0,2,4,6,8\}$ is extreme because it is a five-element subset of the the coset  $\{0,2,4,6,8,10\}$ in \Zp{12}{}.

For the second set, let $\mu =
  \delta_{0}  +   e^{ 3\pi i  / 2}\delta_{2}  +   e^{ 5\pi i  / 4}\delta_{3}  +  \delta_{4}  +   e^{\pi i / 4}\delta_{7}
.$
  Then
\begin{align*} \mu * \tilde\mu&=   5\delta_{0}   +  \big( e^{ 7 \pi i / 4}+ e^{ 3 \pi i / 4}\big)\delta_{1}
  +  \big( -i +  i \big)\delta_{2} 
\\& \qquad
 +  \big( e^{ 5 \pi i / 4}+ e^{\pi i / 4}\big)\delta_{3} 
  +  \big(1-1\big)\delta_{4} 
\\& \qquad +  \big( e^{ 7 \pi i / 4}+ e^{ 3 \pi i / 4}\big)\delta_{5} 
  +  \big( e^{ 5 \pi i / 4}+ e^{\pi i / 4}\big)\delta_{7} 
\\& \qquad +  \big(1-1\big)\delta_{8} 
  +  \big( e^{ 3 \pi i / 4}+ e^{ 7 \pi i / 4}\big)\delta_{9} 
\\& \qquad +  \big(  i - i\big)\delta_{10} 
  +  \big( e^{\pi i / 4}+ e^{ 5 \pi i}\big)\delta_{11} 
  \\&= 5\delta_0.
\end{align*}
\end{proof}

\subsection{Sets with 6 elements}

There are no  extreme three element subsets of \Zp{10}{} nor of \Zp{14}{}, so the sets in the next two results
cannot be the sum of an extreme set with a two-element coset (all cosets having two elements in those groups).

The 6 element set $\{(0, 0),  (0, 1),  (0, 2),  (0, 3),  (0, 4),  (0, 7) \}$ is extreme in $\Zp{2}{} \times \Zp{10}{}$. \newline 
 
\begin{proof}
Let  $\mu =  \delta_{(0,0)}+e^{53\pii/30}\delta_{(0,1)}+e^{8\pii/15}\delta_{(0,2)}+e^{29\pii/30}\delta_{(0,3)}+e^{2\pii/5}\delta_{(0,4)}+e^{\pii/30}\delta_{(0,7)}
$.  Then 
\begin{align*} 
\mu * \tilde \mu & = \big(1+1+1+1+1+1\big)\delta_{(0,0)}
\\&\qquad
+\big(e^{53\pii/30}+e^{23\pii/30}+e^{13\pii/30}+e^{43\pii/30}\big)\delta_{(0,1)}
\\&\qquad
+\big(e^{8\pii/15}+e^{6\pii/5}+e^{28\pii/15}\big)\delta_{(0,2)}
\\&\qquad
+\big(e^{59\pii/30}+e^{29\pii/30}+e^{19\pii/30}+e^{49\pii/30}\big)\delta_{(0,3)}
\\&\qquad
+\big(e^{26\pii/15}+e^{2\pii/5}+e^{16\pii/15}\big)\delta_{(0,4)}
\\&\qquad
+\big(i-i\big)\delta_{(0,5)}
\\&\qquad
+\big(e^{8\pii/5}+e^{14\pii/15}+e^{4\pii/15}\big)\delta_{(0,6)}
\\&\qquad
+\big(e^{31\pii/30}+e^{41\pii/30}+e^{11\pii/30}+e^{\pii/30}\big)\delta_{(0,7)}
\\&\qquad
+\big(e^{22\pii/15}+e^{4\pii/5}+e^{2\pii/15}\big)\delta_{(0,8)}
\\&\qquad
+\big(e^{7\pii/30}+e^{37\pii/30}+e^{47\pii/30}+e^{17\pii/30}\big)\delta_{(0,9)}\\&=6\delta_{(0 , 0)}\qedhere
\end{align*}
\end{proof}

\begin{prop}\label{prop6in10} The 6-element set
$\{0,1,2,3,4,7\} $ is extreme in  \Zp{10}{}.
\end{prop}

\begin{proof}
Let $\mu =
 \delta_{0}  +   e^{ 5\pi i  / 6}\delta_{1}  +   e^{ 2\pi i  / 3}\delta_{2}  +   e^{ 5\pi i  / 6}\delta_{3}  +   e^{ 2\pi i  / 1}\delta_{4}  +   e^{\pi i / 6}\delta_{7}
.$
  Then
\begin{align*} \mu * \tilde\mu&=   6\delta_{0}   +  \big( e^{ 5 \pi i / 6}+ e^{ 11 \pi i / 6}+ e^{\pi i / 6}+ e^{ 7 \pi i / 6}\big)\delta_{1} 
  +  \big( e^{ 2 \pi i / 3}+ 1+ e^{ 4 \pi i / 3}\big)\delta_{2} 
  \\& \qquad 
  +  \big( e^{ 11 \pi i / 6}+ e^{ 5 \pi i / 6}+ e^{ 7 \pi i / 6}+ e^{\pi i / 6}\big)\delta_{3} 
   +  \big( e^{ 2 \pi i / 3}+ 1+ e^{ 4 \pi i / 3}\big)\delta_{4} 
  \\& \qquad
   +  \big(  i - i\big)\delta_{5} 
   +  \big( 1+ e^{ 2 \pi i / 3}+ e^{ 4 \pi i / 3}\big)\delta_{6} 
  \\& \qquad
   +  \big( e^{ 7 \pi i / 6}+ e^{ 5 \pi i / 6}+ e^{ 11 \pi i / 6}+ e^{\pi i / 6}\big)\delta_{7} 
   +  \big( e^{ 4 \pi i / 3}+ 1+ e^{ 2 \pi i / 3}\big)\delta_{8} 
  \\& \qquad +  \big( e^{ 7 \pi i / 6}+ e^{\pi i / 6}+ e^{ 11 \pi i / 6}+ e^{ 5 \pi i}\big)\delta_{9} 
  \\&=6\delta_0.
      \qedhere
\end{align*}
\end{proof}

\begin{prop}\label{prop6in14}
The set $\{ 0,1,2,3,4,7 \}$ is extreme in \Zp{14}{}.
\end{prop}

\begin{proof}
Let
$ \mu =\delta_{0}  +   e^{ 10\pi i  / 12}\delta_{1}  +   e^{ 8\pi i  / 12}\delta_{2}  +   e^{ 10\pi i  / 12}\delta_{3}  +  \delta_{4}  +   e^{ 2\pi i  / 12}\delta_{7}
.$
  Then
\begin{align*}
\mu * \tilde \mu &= 6\delta_{0}   +  \big(  i +  i - i- i\big)\delta_{1} 
  +  \big( e^{\pi i / 1}+ 1+ e^{\pi i / 1}+ 1\big)\delta_{2} 
\\& \qquad +  \big( -i +  i +  i - i\big)\delta_{3} 
  +  \big( 1+ e^{\pi i / 1}+ 1+ e^{\pi i / 1}\big)\delta_{4} 
\\& \qquad +  \big( -i +  i +  i - i\big)\delta_{5} 
  +  \big( e^{\pi i / 1}+ 1+ e^{\pi i / 1}+ 1\big)\delta_{6} 
\\& \qquad +  \big(  i - i+  i +  i - i+  i - i- i\big)\delta_{7} 
\\& \qquad +  \big( e^{\pi i / 1}+ 1+ e^{\pi i / 1}+ 1\big)\delta_{8} 
  +  \big( -i +  i +  i - i\big)\delta_{9} 
  \\&\qquad  +  \big( e^{\pi i / 1}+ 1+ e^{\pi i / 1}+ 1\big)\delta_{10} 
  +  \big( -i - i+  i +  i \big)\delta_{11} 
\\& \qquad
  +  \big( e^{\pi i / 1}+ 1+ e^{\pi i / 1}+ 1\big)\delta_{12} 
  +  \big( -i - i+  i + e^{\pi i / 1}\big)\delta_{13} 
  \\& = 6\delta_0.\qedhere
\end{align*}
\end{proof}

\subsection{Sets with 7 elements}\label{subsec7elts}
Seven-element exceptional extreme sets are somewhat more plentiful; we have
found one in \Zp{12}{}, two non-equivalent ones in \Zp{16}{} (and, as expected, the 7 element subset of the 8-element subgroup)
and one in \Zp{19}{}.

\begin{prop}\label{prop7in8} The 7-element set 
$\{0,1,2,3,4,5,6\}$ is extreme in \Zp8{}
\end{prop}

\begin{proof}
Let  $\mu =  \delta_{0}+e^{4\pii/3}\delta_{1}+e^{4\pii/3}\delta_{2}+e^{2\pii/3}\delta_{3}+e^{4\pii/3}\delta_{4}+e^{4\pii/3}\delta_{5}+\delta_{6}
$.  Then 
\begin{align*} 
\mu * \tilde \mu & = \big(1+1+1+1+1+1+1\big)\delta_{0}
\\&\qquad
+\big(e^{4\pii/3}+1+e^{4\pii/3}+e^{2\pii/3}+1+e^{2\pii/3}\big)\delta_{1}
\\&\qquad
+\big(1+e^{4\pii/3}+e^{4\pii/3}+1+e^{2\pii/3}+e^{2\pii/3}\big)\delta_{2}
\\&\qquad
+\big(e^{2\pii/3}+e^{4\pii/3}+e^{2\pii/3}+1+1+e^{4\pii/3}\big)\delta_{3}
\\&\qquad
+\big(e^{2\pii/3}+1+e^{4\pii/3}+e^{4\pii/3}+1+e^{2\pii/3}\big)\delta_{4}
\\&\qquad
+\big(e^{4\pii/3}+1+1+e^{2\pii/3}+e^{4\pii/3}+e^{2\pii/3}\big)\delta_{5}
\\&\qquad
+\big(e^{2\pii/3}+e^{2\pii/3}+1+e^{4\pii/3}+e^{4\pii/3}+1\big)\delta_{6}
\\&\qquad
+\big(e^{2\pii/3}+1+e^{2\pii/3}+e^{4\pii/3}+1+e^{4\pii/3}\big)\delta_{7}\\&=7\delta_{(0 , 0}\qedhere
\end{align*}
\end{proof}

\begin{rem}
Extreme measures are not unique, even when they are required to have mass 1 at the identity.
 Another extreme measure on $\{0,\dots,6\}\subset \Zp{8}{}$
is \[
\nu=\delta_0+
e^{11\pi i/6}\delta_1+
e^{14\pi i/6} \delta_2
+e^{13\pi i/6}\delta_3
+e^{8\pi i/6}\delta_4
+e^{11\pi i/6}\delta_5
+e^{6\pi i/6}\delta_6.
\]
We omit the proof that $\nu*\widetilde{\nu} =7\delta_0.$
\end{rem}

\begin{prop}\label{prop7in12} The 7-element set
$\{0,1,2,5,6,8,9\}$ is extreme in \Zp{12}{}.
\end{prop}

\begin{proof} Let $\mu =\delta_{0}  +   e^{ 7\pi i  / 12}\delta_{1}  +  \delta_{3}  +   e^{ 5\pi i  / 6}\delta_{4}  +    i \delta_{6}  +   e^{ 7\pi i  / 12}\delta_{7}  +    i \delta_{9}  +   e^{ 5\pi i  / 6}\delta_{10}
$.
Then
\begin{align*} \mu*\tilde \mu  &=7\delta_{0}   +  \big( e^{ 7 \pi i / 12}+ e^{ 5 \pi i / 6}+ e^{\pi i / 12}+ e^{\pi i / 3}\big)\delta_{1} \\& \qquad +  \big( e^{ 7 \pi i / 6}+ e^{ 17 \pi i / 12}+ e^{ 5 \pi i / 3}+ e^{ 23 \pi i / 12}\big)\delta_{2} \\& \qquad +  \big( -i + e^{ 7 \pi i / 4}+ 1+ e^{\pi i / 4}+  i + e^{ 7 \pi i / 4}+ 1+ e^{\pi i / 4}\big)\delta_{3} \\& \qquad +  \big( e^{\pi i / 12}+ e^{ 5 \pi i / 6}+ e^{ 7 \pi i / 12}+ e^{\pi i / 3}\big)\delta_{4} \\& \qquad +  \big( e^{ 17 \pi i / 12}+ e^{ 7 \pi i / 6}+ e^{ 23 \pi i / 12}+ e^{ 5 \pi i / 3}\big)\delta_{5} \\& \qquad +  \big( -i + 1- i+ 1+  i + 1+  i + 1\big)\delta_{6} \\& \qquad +  \big( e^{\pi i / 12}+ e^{\pi i / 3}+ e^{ 7 \pi i / 12}+ e^{ 5 \pi i / 6}\big)\delta_{7} \\& \qquad +  \big( e^{ 7 \pi i / 6}+ e^{ 17 \pi i / 12}+ e^{ 5 \pi i / 3}+ e^{ 23 \pi i / 12}\big)\delta_{8} \\& \qquad +  \big( 1+ e^{ 7 \pi i / 4}- i+ e^{\pi i / 4}+ 1+ e^{ 7 \pi i / 4}+  i + e^{\pi i / 4}\big)\delta_{9} \\& \qquad +  \big( e^{ 7 \pi i / 12}+ e^{\pi i / 3}+ e^{\pi i / 12}+ e^{ 5 \pi i / 6}\big)\delta_{10} \\& \qquad +  \big( e^{ 17 \pi i / 12}+ e^{ 7 \pi i / 6}+ e^{ 23 \pi i / 12}+ e^{ 5 \pi i}\big)\delta_{11} 
\\&=7\delta_0. 
\end{align*}
Each sum in the parenthese above is zero.

\qedhere
\end{proof}

\begin{prop}\label{prop7in16} Each of 7-element sets
\begin{enumerate}
\item
 $\{0,2,4,6,8,10,12\}$ and 
 \item $\{0,1,2,4,5,7,11\}$  
 \end{enumerate}
 is extreme in \Zp{16}{}.
They are not equivalent.
\end{prop}

\begin{proof}
The sets are  not equivalent because the automorphisms 
(multiplication by odd integers) preserve the parity  of
 elements and translation switches or leaves fixed the 
 parity. Since the first set has both odd and even elements
 so will every set equivalent to it. (We will use this argument
 again in \propref{prop8in16}.)
 
 As for extremality, first note that $\{0,2,4,6,8,10,12\}$ is extreme because it is a 7 element subset of an 8 element coset
and so extreme by \propref{prop7in8}. For the other set, set
 \[\mu = 
  \delta_{0}     -\delta_{1}  +   e^{ 4\pi i  / 3}\delta_{2}  +  \delta_{4}  +   e^{ 5\pi i  / 3}\delta_{5}  +   e^{ 5\pi i  / 3}\delta_{7}  +    -\delta_{11}.
\]
Then
\begin{align*}  \mu*\tilde \mu = 7\delta_{0} &  + \big(  -1 + e^{\pi i / 3}+ e^{ 5 \pi i / 3}\big)\delta_{1}   +  \big( e^{ 4 \pi i / 3}+ e^{ 2 \pi i / 3}+ 1\big)\delta_{2} \\&  +  \big(  -1 + e^{\pi i / 3}+ e^{ 5 \pi i / 3}\big)\delta_{3}   +  \big( 1+ e^{ 2 \pi i / 3}+ e^{ 4 \pi i / 3}\big)\delta_{4} \\&  +  \big(  -1 + e^{ 5 \pi i / 3}+ e^{\pi i / 3}\big)\delta_{5}   +  \big( 1+ e^{ 2 \pi i / 3}+ e^{ 4 \pi i / 3}\big)\delta_{6} \\&  +  \big( e^{\pi i / 3}+ e^{ 5 \pi i / 3}- 1 \big)\delta_{7}   +  \big( e^{\pi i / 3}- 1 + e^{ 5 \pi i / 3}\big)\delta_{9} \\&  +  \big( e^{ 4 \pi i / 3}+ e^{ 2 \pi i / 3}+ 1\big)\delta_{10}   +  \big( e^{\pi i / 3}+ e^{ 5 \pi i / 3}- 1 \big)\delta_{11} \\&  +  \big( 1+ e^{ 4 \pi i / 3}+ e^{ 2 \pi i / 3}\big)\delta_{12}   +  \big(  -1 + e^{ 5 \pi i / 3}+ e^{\pi i / 3}\big)\delta_{13} \\&  +  \big( e^{ 2 \pi i / 3}+ e^{ 4 \pi i / 3}+ 1\big)\delta_{14}   +  \big(  -1 + e^{ 5 \pi i / 3}- 1 \big)\delta_{15} \\&=7\delta_0.
\qedhere
\end{align*}
\end{proof}

\begin{prop}\label{prop7in19} The 7-element set
$\{0,1,2,5,12,13,15\}$ is extreme in \Zp{19}{}.
\end{prop}

\begin{proof} 
Let $\mu =\delta_{0}  +   e^{ 4\pi i  / 3}\delta_{1}  +  \delta_{2}  +  \delta_{5}  +   e^{ 4\pi i  / 3}\delta_{12}  +   e^{ 4\pi i  / 3}\delta_{13}  +   e^{\pi i / 3}\delta_{15}.$
Then
\begin{align*} \mu*\tilde \mu &=
7\delta_{0}   +  \big( e^{ 4 \pi i / 3}+ e^{ 2 \pi i / 3}+ 1\big)\delta_{1} 
  +  \big( 1+ e^{\pi i }\big)\delta_{2} 
  \\& \qquad +  \big( 1+ e^{\pi i }\big)\delta_{3} 
   +  \big( e^{ 5 \pi i / 3}+ e^{ 2 \pi i / 3}\big)\delta_{4} \\& \qquad +  \big( e^{\pi i }+ 1\big)\delta_{5} 
   +  \big( e^{ 2 \pi i / 3}+ e^{ 5 \pi i / 3}\big)\delta_{6} \\& \qquad +  \big( e^{ 2 \pi i / 3}+ 1+ e^{ 4 \pi i / 3}\big)\delta_{7} 
    +  \big( 1+ e^{ 2 \pi i / 3}+ e^{ 4 \pi i / 3}\big)\delta_{8} \\& \qquad +  \big( e^{ 2 \pi i / 3}+ e^{ 5 \pi i / 3}\big)\delta_{9}
      +  \big( e^{ 4 \pi i / 3}+ e^{\pi i / 3}\big)\delta_{10} \\& \qquad +  \big( e^{ 2 \pi i / 3}+ 1+ e^{ 4 \pi i / 3}\big)\delta_{11}
 +  \big( e^{ 2 \pi i / 3}+ e^{ 4 \pi i / 3}+ 1\big)\delta_{12} \\& \qquad +  \big( e^{ 4 \pi i / 3}+ e^{\pi i / 3}\big)\delta_{13} 
   +  \big( 1+ e^{\pi i }\big)\delta_{14} \\& \qquad +  \big( e^{ 4 \pi i / 3}+ e^{\pi i / 3}\big)\delta_{15} 
     +  \big( 1+ e^{\pi i }\big)\delta_{16} \\& \qquad +  \big( 1+ e^{\pi i }\big)\delta_{17} 
       +  \big( e^{ 2 \pi i / 3}+ e^{ 4 \pi i / 3}+ 1\big)\delta_{18} 
       \\&= 7\delta_0.  \qedhere
\end{align*}
\end{proof}

We have found no other 7-element extreme sets in cyclic groups of order below 20.
\subsection{Sets with 8 elements}
\begin{prop}\label{prop8inZ9}
The 8 element set $\{0,  1,  2,  3,  4,  5,  6,  7 \}$ is extreme in $\Zp{9}{}$. \newline \end{prop}

\begin{proof}[First proof]
Let  $\mu =  \delta_{0}+e^{10\pii/7}\delta_{1}+e^{4\pii/7}\delta_{2}+e^{2\pii/7}\delta_{3}+e^{2\pii/7}\delta_{4}+e^{4\pii/7}\delta_{5}+e^{10\pii/7}\delta_{6}+\delta_{7}
$.  Then 
\begin{align*} 
\mu * \tilde \mu & = \big(1+1+1+1+1+1+1+1\big)\delta_{0}
\\&\qquad
+\big(e^{10\pii/7}+e^{8\pii/7}+e^{12\pii/7}+1+e^{2\pii/7}+e^{6\pii/7}+e^{4\pii/7}\big)\delta_{1}
\\&\qquad
+\big(1+e^{4\pii/7}+e^{6\pii/7}+e^{12\pii/7}+e^{2\pii/7}+e^{8\pii/7}+e^{10\pii/7}\big)\delta_{2}
\\&\qquad
+\big(e^{4\pii/7}+e^{10\pii/7}+e^{2\pii/7}+e^{6\pii/7}+1+e^{8\pii/7}+e^{12\pii/7}\big)\delta_{3}
\\&\qquad
+\big(e^{10\pii/7}+1+e^{4\pii/7}+e^{2\pii/7}+e^{8\pii/7}+e^{6\pii/7}+e^{12\pii/7}\big)\delta_{4}
\\&\qquad
+\big(e^{12\pii/7}+e^{6\pii/7}+e^{8\pii/7}+e^{2\pii/7}+e^{4\pii/7}+1+e^{10\pii/7}\big)\delta_{5}
\\&\qquad
+\big(e^{12\pii/7}+e^{8\pii/7}+1+e^{6\pii/7}+e^{2\pii/7}+e^{10\pii/7}+e^{4\pii/7}\big)\delta_{6}
\\&\qquad
+\big(e^{10\pii/7}+e^{8\pii/7}+e^{2\pii/7}+e^{12\pii/7}+e^{6\pii/7}+e^{4\pii/7}+1\big)\delta_{7}
\\&\qquad
+\big(e^{4\pii/7}+e^{6\pii/7}+e^{2\pii/7}+1+e^{12\pii/7}+e^{8\pii/7}+e^{10\pii/7}\big)\delta_{8}\\&=8\delta_{0 }\qedhere
\end{align*}
\end{proof}

\begin{proof}[Second proof] Let \begin{align*}\mu =
 \delta_{0}  + &  e^{ 20\pi i  / 21}\delta_{1}  +   e^{ 64\pi i  / 21}\delta_{2}  +   e^{ 20\pi i  / 7}\delta_{3}  +   e^{ 32\pi i  / 21}\delta_{4}  \\ &\qquad +   e^{ 64\pi i  / 21}\delta_{5}  +   e^{ 16\pi i  / 7}\delta_{6}  +   e^{ 8\pi i  / 3}\delta_{7}
.\end{align*}
  Then
\begin{align*} \mu * \tilde\mu&=   8\delta_{0}  +  \big( e^{ 20 \pi i / 21}+ e^{ 2 \pi i / 21}+ e^{ 38 \pi i / 21}+ e^{ 2 \pi i / 3}+ e^{ 32 \pi i / 21}
\\& \qquad \qquad+ e^{ 26 \pi i / 21}+ e^{ 8 \pi i / 21}\big)\delta_{1} 
\\& \qquad
 +  \big( e^{ 4 \pi i / 3}+ e^{ 22 \pi i / 21}+ e^{ 40 \pi i / 21}+ e^{ 10 \pi i / 21}
\\& \qquad \qquad+ e^{ 4 \pi i / 21}+ e^{ 16 \pi i / 21}+ e^{ 34 \pi i / 21}\big)\delta_{2} 
\\& \qquad
 +  \big( e^{ 12 \pi i / 7}+ e^{ 2 \pi i / 7}+ e^{ 6 \pi i / 7}+ e^{ 4 \pi i / 7}
\\& \qquad \qquad+ 1+ e^{ 10 \pi i / 7}+ e^{ 8 \pi i / 7}\big)\delta_{3} 
\\& \qquad
 +  \big( e^{ 20 \pi i / 21}+ e^{ 2 \pi i / 3}+ e^{ 8 \pi i / 21}+ e^{ 32 \pi i / 21}
\\& \qquad \qquad+ e^{ 2 \pi i / 21}+ e^{ 26 \pi i / 21}+ e^{ 38 \pi i / 21}\big)\delta_{4} 
\\& \qquad
 +  \big( e^{ 10 \pi i / 21}+ e^{ 40 \pi i / 21}+ e^{ 16 \pi i / 21}+ e^{ 4 \pi i / 21}
\\& \qquad \qquad+ e^{ 22 \pi i / 21}+ e^{ 4 \pi i / 3}+ e^{ 34 \pi i / 21}\big)\delta_{5} 
\\& \qquad
 +  \big( e^{ 8 \pi i / 7}+ e^{ 10 \pi i / 7}+ 1+ e^{ 4 \pi i / 7}+ e^{ 6 \pi i / 7}+ e^{ 2 \pi i / 7}+ e^{ 12 \pi i / 7}\big)\delta_{6} 
\\& \qquad
 +  \big( e^{ 20 \pi i / 21}+ e^{ 2 \pi i / 21}+ e^{ 32 \pi i / 21}+ e^{ 38 \pi i / 21}
\\& \qquad \qquad+ e^{ 26 \pi i / 21}+ e^{ 8 \pi i / 21}+ e^{ 2 \pi i / 3}\big)\delta_{7} 
\\& \qquad +  \big( e^{ 22 \pi i / 21}+ e^{ 40 \pi i / 21}+ e^{ 4 \pi i / 21}
\\& \qquad \qquad+ e^{ 4 \pi i / 3}+ e^{ 10 \pi i / 21}+ e^{ 16 \pi i / 21}+ e^{ 34 \pi i}\big)\delta_{8} 
\\&= 9\delta_0.
      \qedhere
\end{align*}
\end{proof}



\begin{prop}\label{prop8in10} The 8-element set
$\{0,1,2,3,5,6,7,8\}$ is extreme in \Zp{10}{}. 
\end{prop}

\begin{proof}[Computational proof of \propref{prop8in10}] 
Let $\mu=\delta_{0}  +   e^{ 7\pi i  / 6}\delta_{1}  +   e^{ 2\pi i  / 3}\delta_{2}  +    i \delta_{3}  +    i \delta_{5}  +   e^{ 2\pi i  / 3}\delta_{6}  +   e^{ 7\pi i  / 6}\delta_{7}  +  \delta_{8} $.
Then
\begin{align*}  
\mu*\tilde\mu &=   8\delta_{0}  +  \big( e^{ 7 \pi i / 6}- i+ e^{ 11 \pi i / 6}+ e^{\pi i / 6}+  i + e^{ 5 \pi i / 6}\big)\delta_{1} 
\\& \qquad +  \big( 1+ e^{ 2 \pi i / 3}+ e^{ 4 \pi i / 3}+ 1+ e^{ 2 \pi i / 3}+ e^{ 4 \pi i / 3}\big)\delta_{2} 
\\& \qquad +  \big( e^{ 5 \pi i / 6}+ e^{ 7 \pi i / 6}+  i + e^{ 11 \pi i / 6}+ e^{\pi i / 6}- i\big)\delta_{3} 
\\& \qquad +  \big( e^{ 4 \pi i / 3}+ 1+ e^{ 2 \pi i / 3}+ e^{ 4 \pi i / 3}+ 1+ e^{ 2 \pi i / 3}\big)\delta_{4} 
\\& \qquad +  \big( -i +  i - i+  i +  i - i+  i - i\big)\delta_{5} 
\\& \qquad +  \big( e^{ 2 \pi i / 3}+ 1+ e^{ 4 \pi i / 3}+ e^{ 2 \pi i / 3}+ 1+ e^{ 4 \pi i / 3}\big)\delta_{6} 
\\& \qquad +  \big( -i + e^{\pi i / 6}+ e^{ 11 \pi i / 6}+  i + e^{ 7 \pi i / 6}+ e^{ 5 \pi i / 6}\big)\delta_{7} 
\\& \qquad +  \big( e^{ 4 \pi i / 3}+ e^{ 2 \pi i / 3}+ 1+ e^{ 4 \pi i / 3}+ e^{ 2 \pi i / 3}+ 1\big)\delta_{8} 
\\& \qquad +  \big( e^{ 5 \pi i / 6}+  i + e^{\pi i / 6}+ e^{ 11 \pi i / 6}- i+ e^{ 7 \pi i }\big)\delta_{9}
\\&=9\delta_0.\qedhere
\end{align*}
\end{proof}

\begin{proof}[Direct proof of \propref{prop8in10}] The set is the sum of a 4-element set in the 5-element
subgroup $\{0,2,4,6,8\}$ of \Zp{10}{} and the subgroup $\{0,5\}$:
$\{0,1,2,3,5,6,7,8\} = \{0,2,6,8\}+\{0,5\}=\{0,2,6,8,5,7,11,13\} = \{0,2,6,8,5,7,1,3\}$ $\mod 10$.
\end{proof}

Turning to \Zp{12}{}, we observe that an 8-element set there could be a sum of a 4-element set and a 2-element set.
However, we cannot get an extreme set that way, though there is an 8-element extreme set in \Zp{12}{}, as the
following Lemma and Proposition show.

\begin{lem}\label{lemsumsinZtwlv}
\Zp{12}{} has no 8-element set that is a sum of a four-element extreme set and a coset.
\end{lem}

\begin{proof}
Suppose $E= A+B$ has 8 elements, where $A$ has 4 elements and $B$ has two. We may assume $0\in A$ and $B=\{0,6\}$.
Clearly $A$ cannot contain 6. \refitem{remitOneRepGen} and calculation (both by hand and machine) show that if a 4-element set in \Zp{12}{} 
 lacks 6, it is not 
extreme.\footnote{\,
The only 4-element extreme sets of \Zp{12}{} (up to equivalence) 
are $\{0,1,6,7\},$ $\{0, 2,6,8\}$ and $\{0,3,6,9\}$, all containing 6. Several other four-element sets pass the test of
\refitem{remitOneRepGen} but are not extreme.}
\end{proof}

\begin{prop}\label{prop8in12} The 8-element set
$\{0,1,3,4,6,7,9,10\}$ 
 is extreme in \Zp{12}{} and is not the sum of two extreme sets in spite of being the sum $ \{0,1,3,4\} +\{0,6\}$
 and also the sum $\{0,3,6,9\}+\{0,1\}$. 
\end{prop}
\begin{proof}Let $\mu =  \delta_{0}  +   e^{ 7\pi i  / 6}\delta_{1}  +  \delta_{3}  +   e^{ 5\pi i  / 3}\delta_{4}  +   e^{\pi i }\delta_{6}  +   e^{ 7\pi i  / 6}\delta_{7}  +   e^{\pi i }\delta_{9}  +   e^{ 5\pi i  / 3}\delta_{10}$.

Indeed, each of the sums in parentheses below is a net zero:
\begin{align*} \mu*\tilde \mu = 8\delta_{0} &   +  \big( e^{ 7 \pi i / 6}+ e^{ 5 \pi i / 3}+ e^{\pi i / 6}+ e^{ 2 \pi i / 3}\big)\delta_{1} 
\\&
 +  \big( e^{\pi i / 3}+ e^{ 5 \pi i / 6}
+ e^{ 4 \pi i / 3}+ e^{ 11 \pi i / 6}\big)\delta_{2} 
\\
&  +  \big( e^{\pi i }- i+ 1+  i + e^{\pi i }- i+ 1+  i \big)\delta_{3} 
\\
&  +  \big( e^{\pi i / 6}+ e^{ 5 \pi i / 3}+ e^{ 7 \pi i / 6}+ e^{ 2 \pi i / 3}\big)\delta_{4} 
\\&
  +  \big( e^{ 5 \pi i / 6}+ e^{\pi i / 3}+ e^{ 11 \pi i / 6}+ e^{ 4 \pi i / 3}\big)\delta_{5} 
\\
&  +  \big( e^{\pi i }+ 1+ e^{\pi i }+ 1+ e^{\pi i }+ 1+ e^{\pi i }+ 1\big)\delta_{6} 
\\&
  +  \big( e^{\pi i / 6}+ e^{ 2 \pi i / 3}+ e^{ 7 \pi i / 6}+ e^{ 5 \pi i / 3}\big)\delta_{7} 
\\
&  +  \big( e^{\pi i / 3}+ e^{ 5 \pi i / 6}+ e^{ 4 \pi i / 3}+ e^{ 11 \pi i / 6}\big)\delta_{8} 
\\
&  +  \big( 1- i+ e^{\pi i }+  i + 1- i+ e^{\pi i }+  i \big)\delta_{9} 
\\&
  +  \big( e^{ 7 \pi i / 6}+ e^{ 2 \pi i / 3}+ e^{\pi i / 6}+ e^{ 5 \pi i / 3}\big)\delta_{10} 
\\&
  +  \big( e^{ 5 \pi i / 6}+ e^{\pi i / 3}+ e^{ 11 \pi i / 6}+ e^{ 4 \pi i }\big)\delta_{11}
  \\& = 8\delta_0. 
\end{align*}
Now apply \lemref{lemsumsinZtwlv} and the fact that $\{0,1\}$ is not extreme in \Zp{12}{}\footnote{The difference 
$\{0,1\}-\{0,1\}$ produces the terms $0-0$, 1-1, 0-1, 1-0 and so $\{0,1\} \subset \Zp{12}{} $ fails the test of \refitem{remitOneRepGen}.}
\end{proof}


\begin{prop}\label{prop8in16} Each of the 8-element sets
\begin{enumerate}
\item
$\{0, 1, 4, 5, 8, 9, 12, 13\}$ 
\item $\{0,2,4,6,8,10,12,14\}$ 
\end{enumerate}
is  extreme in \Zp{16}{}. They are not equivalent.
\end{prop}

\begin{proof}
The sets are  not equivalent because the autormorphisms 
(multiplication by odd integers) preserve the parity  of
 elements and translation switches or leaves fixed the 
 parity. Since the first set has both odd and even elements, 
 so will every set equivalent to it. 
 
 The second set is extreme because it is a coset.
 
For the extremality of the first set, let $\mu = \delta_{0} + e^{ 22 \pi i / 12}\delta_{1} + e^{ 18 \pi i / 12}\delta_{4} + e^{ 22 \pi i / 12}\delta_{5} +\delta_{8} + e^{ 10 \pi i / 12}\delta_{9} + e^{ 18 \pi i / 12}\delta_{12} + e^{ 10 \pi i / 12}\delta_{13} 
$.  Then 
\begin{align*} 
\mu * \tilde \mu & = 8\delta_{0}   +  \big( e^{ 11 \pi i / 6}+ e^{\pi i / 3}+ e^{ 5 \pi i / 6}+ e^{ 4 \pi i / 3}\big)\delta_{1}
\\&\qquad  
+  \big( e^{ 7 \pi i / 6}+ e^{ 5 \pi i / 3}+ e^{\pi i / 6}+ e^{ 2 \pi i / 3}\big)\delta_{3} \\&\qquad  +  \big( e^{\pi i / 2}+ e^{\pi i}+ e^{ 3 \pi i / 2}+ 1+ e^{\pi i / 2}+ e^{\pi i}+ e^{ 3 \pi i / 2}+ 1\big)\delta_{4} \\&\qquad  +  \big( e^{\pi i / 3}+ e^{ 11 \pi i / 6}+ e^{ 4 \pi i / 3}+ e^{ 5 \pi i / 6}\big)\delta_{5}
 \\&\qquad  
 +  \big( e^{ 7 \pi i / 6}+ e^{ 2 \pi i / 3}+ e^{\pi i / 6}+ e^{ 5 \pi i / 3}\big)\delta_{7} \\&\qquad  +  \big( 1+ e^{\pi i}+ 1+ e^{\pi i}+ 1+ e^{\pi i}+ 1+ e^{\pi i}\big)\delta_{8} \\&\qquad  +  \big( e^{ 11 \pi i / 6}+ e^{\pi i / 3}+ e^{ 5 \pi i / 6}+ e^{ 4 \pi i / 3}\big)\delta_{9} \\&\qquad  +  \big( e^{\pi i / 6}+ e^{ 2 \pi i / 3}+ e^{ 7 \pi i / 6}+ e^{ 5 \pi i / 3}\big)\delta_{11} \\&\qquad  +  \big( e^{\pi i / 2}+ 1+ e^{ 3 \pi i / 2}+ e^{\pi i}+ e^{\pi i / 2}+ 1+ e^{ 3 \pi i / 2}+ e^{\pi i}\big)\delta_{12} \\&\qquad  +  \big( e^{\pi i / 3}+ e^{ 11 \pi i / 6}+ e^{ 4 \pi i / 3}+ e^{ 5 \pi i / 6}\big)\delta_{13} \\&\qquad  +  \big( e^{\pi i / 6}+ e^{ 5 \pi i / 3}+ e^{ 7 \pi i / 6}+ e^{ 2 \pi i}\big)\delta_{15} 
\\&=8\delta_0.\qedhere
\end{align*}
\end{proof}

\break

\subsection{Sets with 9 elements}

\begin{prop}\label{prop9in10}
$E=\{0,1,2,3,4,5,6,7,8\}$ is extreme in \Zp{10}{}.
\end{prop}

\begin{proof} Let $\mu =
 \delta_{0}  +   e^{ 3\pi i  / 2}\delta_{1}  +  \delta_{2}  +    i \delta_{3}  +    -\delta_{4}  +    i \delta_{5}  +  \delta_{6}  +   e^{ 3\pi i  / 2}\delta_{7}  +  \delta_{8}
.$
  Then
\begin{align*} \mu * \tilde\mu&=   9\delta_{0}     +  \big( -i +  i +  i +  i - i- i- i+  i \big)\delta_{1} \\& \qquad +  \big( 1+1-1- 1 +1-1- 1 + 1\big)\delta_{2} \\& \qquad +  \big(  i - i+  i - i+  i - i+  i - i\big)\delta_{3} \\& \qquad +  \big( 1+ 1+1-1- 1 +1-1- 1 \big)\delta_{4} \\& \qquad +  \big( -i - i+  i +  i +  i +  i - i- i\big)\delta_{5} \\& \qquad +  \big(  -1 - 1 +1-1- 1 + 1+ 1+ 1\big)\delta_{6} \\& \qquad +  \big( -i +  i - i+  i - i+  i - i+  i \big)\delta_{7} \\& \qquad +  \big(1-1- 1 +1-1- 1 + 1+ 1\big)\delta_{8} \\& \qquad +  \big(  i - i- i- i+  i +  i +  i + e^{ 3 \pi i}\big)\delta_{9}
\\&= 9\delta_0.
\end{align*}

\end{proof}

\begin{prop}\label{prop9in12}
$\{0,1,2,3,4,5,6,7,8\}$ and $\{0,1,2,4,5,6,8,9,10\}$ are extreme in \Zp{12}{}. They are  non-equivalent. 
\end{prop}

\begin{proof}
There are two ways to prove non-equivalence. First, by  a tedious calculation (which we delegated to a computer). The second is to show that one of the sets is the sum
of a coset with a three-element set and the other is not, which we do in the next paragraph.

The only 3-element subgroup in \Zp{12}{} is $\{0,4,8\}$. If $\{0,4,8\}+\{a\} 
\subset\{0,1,2,3,4,5,6,7,8\}$, then $a\ne\, 1,2,3$,
 since $8+1=9,8+2=10, 8+3=11$ are 
  not in $E.$  Similarly,
   $a\ne 5,6,7,9,10,11$
 since none of those elements is in $E$. Thus, $a\in \{0,4,8\}$ and the first set is seen not be a sum.
 
 Of course, $\{0,4,8\}+\{0,1,2\}=\{0,1,2,4,5,6,8,9,10\}$ and both final conclusions follow.

(1). For the extremality of the first set, let $\mu =
 \delta_{0}  +   e^{ 23\pi i  / 12}\delta_{1}  +   e^{ 3\pi i  / 2}\delta_{2}  +   e^{ 17\pi i  / 12}\delta_{3}  +  \delta_{4}  +   e^{\pi i / 4}\delta_{5}  +   e^{ 7\pi i  / 6}\delta_{6}  +   e^{ 5\pi i  / 12}\delta_{7}  +   e^{ 4\pi i  / 3}\delta_{8}
.$
  Then
\begin{align*}\mu*\tilde\mu &= 9\delta_0+ \big( e^{ 23 \pi i / 12}+ e^{ 19 \pi i / 12}+ e^{ 23 \pi i / 12}+ e^{ 7 \pi i / 12}+ e^{\pi i / 4}
\\&\quad\qquad+ e^{ 11 \pi i / 12}+ e^{ 5 \pi i / 4}+ e^{ 11 \pi i / 12}\big)\delta_{1} 
\\&\qquad  +  \big( -i - i+  i + e^{ 5 \pi i / 6}+ e^{ 7 \pi i / 6}+ e^{\pi i / 6}+ e^{\pi i / 6}\big)\delta_{2} 
\\&\qquad  +  \big( e^{ 17 \pi i / 12}+ e^{\pi i / 12}+ e^{ 3 \pi i / 4}+ e^{ 7 \pi i / 4}+ e^{ 5 \pi i / 12}+ e^{ 13 \pi i / 12}\big)\delta_{3}
 \\&\qquad  +  \big( e^{ 2 \pi i / 3}+ 1+ e^{\pi i / 3}+ e^{ 5 \pi i / 3}+ e^{ 1 \pi i}+ e^{ 4 \pi i / 3}\big)\delta_{4} 
\\&\qquad  +  \big( e^{ 19 \pi i / 12}+ e^{ 7 \pi i / 12}+ e^{\pi i / 4}+ e^{ 5 \pi i / 4}+ e^{ 11 \pi i / 12}+ e^{ 23 \pi i / 12}\big)\delta_{5} 
\\&\qquad  +  \big( e^{ 5 \pi i / 6}- i+ e^{\pi i / 6}+ e^{ 7 \pi i / 6}+  i + e^{ 11 \pi i / 6}\big)\delta_{6} 
\\&\qquad  +  \big( e^{ 7 \pi i / 4}+ e^{ 3 \pi i / 4}+ e^{ 13 \pi i / 12}+ e^{\pi i / 12}+ e^{ 5 \pi i / 12}+ e^{ 17 \pi i / 12}\big)\delta_{7} 
\\&\qquad  +  \big( 1+ e^{ 5 \pi i / 3}+ e^{\pi i / 3}+ e^{ 1 \pi i}+ e^{ 2 \pi i / 3}+ e^{ 4 \pi i / 3}\big)\delta_{8} 
\\&\qquad  +  \big( e^{ 7 \pi i / 12}+ e^{ 23 \pi i / 12}+ e^{ 5 \pi i / 4}+ e^{\pi i / 4}+ e^{ 19 \pi i / 12}+ e^{ 11 \pi i / 12}\big)\delta_{9} 
\\&\qquad  +  \big(  i +  i - i+ e^{ 7 \pi i / 6}+ e^{ 5 \pi i / 6}+ e^{ 11 \pi i / 6}+ e^{ 11 \pi i / 6}\big)\delta_{10} 
\\&\qquad  +  \big( e^{\pi i / 12}+ e^{ 5 \pi i / 12}+ e^{\pi i / 12}+ e^{ 17 \pi i / 12}+ e^{ 7 \pi i / 4}
\\&\quad\qquad+ e^{ 13 \pi i / 12}+ e^{ 3 \pi i / 4}+ e^{ 13 \pi i}\big)\delta_{11} 
\\&=9\delta_0.
\end{align*}

(2).  Let $  
\mu = \delta_{0}  +   e^{ 10\pi i  / 6}\delta_{1}  +   e^{ 10\pi i  / 6}\delta_{2}  +   e^{ 8\pi i  / 6}\delta_{4}  +   e^{ 102\pi i  / 6}\delta_{5}  +   e^{ 2\pi i  / 6}\delta_{6}  +   e^{ 4\pi i  / 6}\delta_{8}  +   e^{ 10\pi i  / 6}\delta_{9}  +   e^{ 6\pi i  / 6}\delta_{10}
.$
  Then
\begin{align*}\mu * \tilde\mu&=    9\delta_{0}   +  \big( e^{ 5 \pi i / 3}+ 1+ e^{ 5 \pi i / 3}+ e^{ 4 \pi i / 3}+ e^{\pi  i}+ e^{ 4 \pi i / 3}\big)\delta_{1} \\&\qquad  +  \big( e^{\pi  i}+ e^{ 5 \pi i / 3}+ e^{ 5 \pi i / 3}+ e^{\pi  i}+ e^{\pi i / 3}+ e^{\pi i / 3}\big)\delta_{2} 
\\&\qquad  +  \big( e^{\pi i / 3}+ e^{ 2 \pi i / 3}+ e^{ 5 \pi i / 3}+ e^{ 4 \pi i / 3}+ e^{ 5 \pi i / 3}+ e^{ 4 \pi i / 3}\big)\delta_{3} \\&\qquad  +  \big( e^{ 4 \pi i / 3}+ 1+ e^{ 2 \pi i / 3}+ e^{ 4 \pi i / 3}+ e^{ 4 \pi i / 3}+ e^{ 2 \pi i / 3}+ e^{ 4 \pi i / 3}
\\& \quad\qquad + e^{ 2 \pi i / 3}+ e^{ 2 \pi i / 3}\big)\delta_{4} 
\\&\qquad  +  \big( e^{\pi  i}+ 1+ e^{\pi  i}+ e^{ 2 \pi i / 3}+ e^{\pi i / 3}+ 1\big)\delta_{5} \\&\qquad  +  \big( e^{ 5 \pi i / 3}+ e^{\pi  i}+ e^{\pi i / 3}+ e^{\pi i / 3}+ e^{\pi  i}+ e^{ 5 \pi i / 3}\big)\delta_{6} 
\\&\qquad  +  \big( e^{\pi  i}+ e^{ 4 \pi i / 3}+ e^{ 5 \pi i / 3}+ 1+ e^{\pi  i}+ 1\big)\delta_{7} \\&\qquad  +  \big( e^{ 2 \pi i / 3}+ e^{ 2 \pi i / 3}+ e^{ 4 \pi i / 3}+ e^{ 2 \pi i / 3}+ e^{ 4 \pi i / 3}\\& \quad\qquad + e^{ 4 \pi i / 3}+ e^{ 2 \pi i / 3}+ 1+ e^{ 4 \pi i / 3}\big)\delta_{8} 
\end{align*}
\begin{align*}
\phantom{\mu * \tilde\mu}&\qquad  +  \big( e^{\pi i / 3}+ e^{ 2 \pi i / 3}+ e^{\pi i / 3}+ e^{ 2 \pi i / 3}+ e^{ 5 \pi i / 3}+ e^{ 4 \pi i / 3}\big)\delta_{9} \\&\qquad  +  \big( e^{\pi i / 3}+ e^{\pi i / 3}+ e^{\pi  i}+ e^{ 5 \pi i / 3}+ e^{ 5 \pi i / 3}+ e^{\pi  i}\big)\delta_{10} \\&\qquad  +  \big( e^{\pi i / 3}+ 1+ e^{\pi i / 3}+ e^{ 2 \pi i / 3}+ e^{\pi  i}+ e^{ 2 \pi i}\big)\delta_{11} 
   \\&=9\delta_0. 
\end{align*}
\end{proof}

\begin{prop}\label{prop9in13} The 9-element set
$\{0, 1, 2, 3, 4, 5, 7, 9, 10\}$ is extreme in \Zp{13}{}.
\end{prop}

\begin{proof}
Let $\mu = \delta_{0} +\delta_{1} - \delta_{2} - \delta_{3} +\delta_{4} - \delta_{5} - \delta_{7} - \delta_{9} - \delta_{10} 
$.  Then 
\begin{align*} 
\mu * \tilde \mu & = 9\delta_{0}   +  \big(1-1+1-1- 1 + 1\big)\delta_{1} 
 +  \big(  -1 - 1 - 1 + 1+ 1+ 1\big)\delta_{2} \\&\qquad  +  \big(  -1 - 1 + 1+1-1+ 1\big)\delta_{3} 
  +  \big(  -1 - 1 +1-1+ 1+ 1\big)\delta_{4} \\&\qquad  +  \big(  -1 +1-1+1-1+ 1\big)\delta_{5} 
   +  \big(  -1 + 1+1-1+1-1\big)\delta_{6} \\&\qquad  +  \big(  -1 +1-1- 1 + 1+ 1\big)\delta_{7} 
    +  \big(  -1 +1-1+1-1+ 1\big)\delta_{8} \\&\qquad  +  \big(1-1+ 1+1-1- 1 \big)\delta_{9} 
     +  \big(  -1 + 1+1-1+1-1\big)\delta_{10} \\&\qquad  +  \big(  -1 - 1 - 1 + 1+ 1+ 1\big)\delta_{11} \
    +  \big(1-1+1-1- 1 + 1\big)\delta_{12} 
\\&=9\delta_0.\qedhere
\end{align*}
\end{proof}

\subsection{Sets with 10 elements}
The only 10-element extreme sets we have found are sums of 5-element cosets with 2-element cosets and 5-element seubsets of
6-element cosets with 2-element cosets.

 \subsection{Sets with 11 elements}
 
\begin{prop}\label{prop11inZ12}
The 11-element set $\{0,  1,  2,  3,  4,  5,  6,  7,  8,  9,  10 \}$ is extreme in $\Zp{12}{}$. 
\newline \end{prop}

\begin{proof}
Let  $\mu =  \delta_{0}+e^{7\pii/4}\delta_{1}-i\delta_{2}+e^{5\pii/4}\delta_{3}+\delta_{4}+e^{7\pii/4}\delta_{5}+i\delta_{6}+e^{\pii/4}\delta_{7}-\delta_{8}+e^{7\pii/4}\delta_{9}+i\delta_{10}
$.  Then 
\begin{align*} 
\mu * \tilde \mu & = \big(1+1+1+1+1+1+1+1+1+1+1\big)\delta_{0}
\\&\qquad
+\big(e^{7\pii/4}+e^{7\pii/4}+e^{7\pii/4}+e^{3\pii/4}+e^{7\pii/4}+e^{3\pii/4}
\\&\qquad\qquad +e^{7\pii/4}+e^{3\pii/4}+e^{3\pii/4}+e^{3\pii/4}-i-i-i+i+
\\&\qquad\qquad
i+i+i+i-i-i+e^{\pii/4}
+e^{5\pii/4}+e^{5\pii/4}+e^{\pii/4}+e^{\pii/4}+e^{5\pii/4}
\\&\qquad\qquad
+e^{\pii/4}+e^{5\pii/4}+e^{5\pii/4}
\\&\qquad\qquad
+e^{\pii/4}-1+1-1+1+1-1-1-1+1+1
\\&\qquad\qquad
+e^{7\pii/4}+e^{3\pii/4}+e^{7\pii/4}+e^{3\pii/4}+e^{7\pii/4}+e^{3\pii/4}+e^{3\pii/4}
\\&\qquad\qquad+e^{7\pii/4}+e^{7\pii/4}+e^{3\pii/4}-i-i+i-i-i+i+i-i+i+i\big)\delta_{6}
\\&\qquad
+\big(e^{\pii/4}+e^{5\pii/4}+e^{5\pii/4}+e^{\pii/4}+e^{\pii/4}+e^{5\pii/4}
\\&\qquad\qquad
+e^{\pii/4}+e^{5\pii/4}+e^{\pii/4}+e^{5\pii/4}\big)\delta_{7}
\\&\qquad
+\big(1+1-1-1-1+1+1-1+1-1\big)\delta_{8}
\\&\qquad
+\big(e^{3\pii/4}+e^{7\pii/4}+e^{7\pii/4}+e^{3\pii/4}+e^{7\pii/4}+e^{3\pii/4}+e^{3\pii/4}
\\&\qquad\qquad+e^{7\pii/4}+e^{7\pii/4}+e^{3\pii/4}\big)\delta_{9}
\\&\qquad
+\big(i+i-i-i-i-i-i+i+i+i\big)\delta_{10}
\\&\qquad
+\big(e^{\pii/4}+e^{\pii/4}+e^{\pii/4}+e^{5\pii/4}+e^{\pii/4}+e^{5\pii/4}+e^{\pii/4}
\\&\qquad\qquad+e^{5\pii/4}+e^{5\pii/4}+e^{5\pii/4}\big)\delta_{11}\\&=11\delta_{0 }\qedhere
\end{align*}
\end{proof}

\subsection{Sets with 12 elements}
 
The only 12-element extreme sets we have found are sums of 6-element cosets with 2-element cosets.

 \subsection{Sets with 13 elements}

$ \delta_{0}+\delta_{12}+e^{4\pii/3}(\delta_{1}+ \delta_{3} + \delta_{4} +\delta_{8}+ \delta_{9}+\delta_{11})
+e^{2\pii/3}(\delta_{2}+ \delta_{5}+ \delta_{6}+ \delta_{7}+ \delta_{10})$

Let  $\mu =  \delta_{(0,0)}+e^{4\pii/3}\delta_{(0,1)}+e^{2\pii/3}\delta_{(0,2)}+e^{4\pii/3}\delta_{(0,3)}+e^{4\pii/3}\delta_{(0,4)}+e^{2\pii/3}\delta_{(0,5)}+e^{2\pii/3}\delta_{(0,6)}+e^{2\pii/3}\delta_{(0,7)}+e^{4\pii/3}\delta_{(0,8)}+e^{4\pii/3}\delta_{(0,9)}+e^{2\pii/3}\delta_{(0,10)}+e^{4\pii/3}\delta_{(0,11)}+\delta_{(0,12)}
$.  Then \begin{align*} 
\mu * \tilde \mu & = \big(1+1+1+1+1+1+1+1+1+1+1+1+1\big)\delta_{(0,0)}
\\&\qquad
+\big(e^{4\pii/3}+e^{4\pii/3}+e^{2\pii/3}+1+e^{4\pii/3}+1+1+e^{2\pii/3}+1+e^{4\pii/3}+e^{2\pii/3}+e^{2\pii/3}\big)\delta_{(0,1)}
\\&\qquad
+\big(1+e^{2\pii/3}+1+e^{2\pii/3}+e^{4\pii/3}+e^{4\pii/3}+1+e^{2\pii/3}+e^{2\pii/3}+e^{4\pii/3}+1+e^{4\pii/3}\big)\delta_{(0,2)}
\\&\qquad
+\big(e^{2\pii/3}+e^{4\pii/3}+e^{4\pii/3}+1+1+e^{4\pii/3}+e^{4\pii/3}+e^{2\pii/3}+e^{2\pii/3}+1+1+e^{2\pii/3}\big)\delta_{(0,3)}
\\&\qquad
+\big(e^{4\pii/3}+1+e^{2\pii/3}+e^{4\pii/3}+e^{4\pii/3}+1+e^{4\pii/3}+1+e^{2\pii/3}+1+e^{2\pii/3}+e^{2\pii/3}\big)\delta_{(0,4)}
\\&\qquad
+\big(e^{2\pii/3}+e^{2\pii/3}+e^{4\pii/3}+e^{4\pii/3}+e^{2\pii/3}+e^{4\pii/3}+1+1+1+1+e^{2\pii/3}+e^{4\pii/3}\big)\delta_{(0,5)}
\\&\qquad
+\big(e^{2\pii/3}+1+1+1+e^{4\pii/3}+e^{2\pii/3}+e^{4\pii/3}+e^{2\pii/3}+1+e^{4\pii/3}+e^{2\pii/3}+e^{4\pii/3}\big)\delta_{(0,6)}
\\&\qquad
+\big(e^{4\pii/3}+1+e^{4\pii/3}+e^{2\pii/3}+1+e^{2\pii/3}+e^{2\pii/3}+1+e^{2\pii/3}+e^{4\pii/3}+1+e^{4\pii/3}\big)\delta_{(0,7)}
\\&\qquad
+\big(e^{4\pii/3}+e^{2\pii/3}+e^{4\pii/3}+1+e^{2\pii/3}+e^{4\pii/3}+e^{2\pii/3}+e^{4\pii/3}+1+1+1+e^{2\pii/3}\big)\delta_{(0,8)}
\\&\qquad
+\big(e^{4\pii/3}+e^{2\pii/3}+1+1+1+1+e^{4\pii/3}+e^{2\pii/3}+e^{4\pii/3}+e^{4\pii/3}+e^{2\pii/3}+e^{2\pii/3}\big)\delta_{(0,9)}
\\&\qquad
+\big(e^{2\pii/3}+e^{2\pii/3}+1+e^{2\pii/3}+1+e^{4\pii/3}+1+e^{4\pii/3}+e^{4\pii/3}+e^{2\pii/3}+1+e^{4\pii/3}\big)\delta_{(0,10)}
\\&\qquad
+\big(e^{2\pii/3}+1+1+e^{2\pii/3}+e^{2\pii/3}+e^{4\pii/3}+e^{4\pii/3}+1+1+e^{4\pii/3}+e^{4\pii/3}+e^{2\pii/3}\big)\delta_{(0,11)}
\\&\qquad
+\big(e^{4\pii/3}+1+e^{4\pii/3}+e^{2\pii/3}+e^{2\pii/3}+1+e^{4\pii/3}+e^{4\pii/3}+e^{2\pii/3}+1+e^{2\pii/3}+1\big)\delta_{(0,12)}
\\&\qquad
+\big(e^{2\pii/3}+e^{2\pii/3}+e^{4\pii/3}+1+e^{2\pii/3}+1+1+e^{4\pii/3}+1+e^{2\pii/3}+e^{4\pii/3}+e^{4\pii/3}\big)\delta_{(0,13)}\\&=13\delta_{(0 , 0)}\qedhere
\end{align*}

 \begin{prop}\label{prop13in14}
  $\{0-12\}$ is extreme in \Zp{14}{}.
 \end{prop}
 xxx
 \begin{proof} Let 
$\mu =
  \delta_{0}  +   e^{ 7\pi i  / 4}\delta_{1}  +   \delta_{2}  +   e^{ 3\pi i  / 4}\delta_{3}  +   e^{ 3\pi i  / 2}\delta_{4}  +   e^{ 5\pi i  / 4}\delta_{5}  
+   i\delta_{6}  +   e^{ \pi i  / 4}\delta_{7}  +   e^{ 3\pi i  / 4}\delta_{8}  +   e^{ 7\pi i  / 4}\delta_{9}  +  \delta_{10}  +   e^{ 3\pi i  / 4}\delta_{11}  +   \delta_{12}.
$

  Then
\begin{align*}\mu * \tilde\mu&=   13\delta_{0}   +  \big( e^{ 7 \pi i / 4}+ e^{ 9 \pi i / 4}+ e^{ 3 \pi i / 4}+ e^{ 11 \pi i / 4}+ e^{ 7 \pi i / 4}+ e^{ 7 \pi i / 4}
\\& \quad\qquad + e^{ 5 \pi i / 4}+ e^{ 13 \pi i / 4}+ e^{ 9 \pi i / 4}+ e^{ 9 \pi i / 4}+ e^{ 3 \pi i / 4}+ e^{ 13 \pi i / 4}\big)\delta_{1} 
\\& \qquad  +  \big( 1+ 1+ e^{ \pi  i}- i+ e^{ 5 \pi i / 2}- i+ e^{ \pi  i}
\\& \quad\qquad + e^{ 5 \pi i / 2}+ e^{ 7 \pi i / 2}+ e^{ 5 \pi i / 2}+ e^{ \pi  i}+ 1\big)\delta_{2} 
\\& \qquad  +  \big( e^{ 5 \pi i / 4}+ e^{ 7 \pi i / 4}+ e^{ 3 \pi i / 4}+ e^{ 7 \pi i / 4}+ e^{ 5 \pi i / 4}+ e^{ 9 \pi i / 4}+ e^{ 3 \pi i / 4}
\\& \quad\qquad+ e^{ 9 \pi i / 4}+ e^{ 11 \pi i / 4}+ e^{ 15 \pi i / 4}+ e^{ 5 \pi i / 4}+ e^{ 9 \pi i / 4}\big)\delta_{3} 
\\& \qquad  +  \big( 1+ e^{ 3 \pi i}+ 1- i- i+ e^{ \pi  i}- i+ 1
\\& \quad\qquad
+ e^{ 5 \pi i / 2}+ e^{ 3 \pi i}+ e^{ 5 \pi i / 2}+ e^{ 5 \pi i / 2}\big)\delta_{4} 
\\& \qquad  +  \big( e^{\pi i / 4}+ e^{ 7 \pi i / 4}+ e^{ 13 \pi i / 4}+ e^{ 3 \pi i / 4}+ e^{ 5 \pi i / 4}+ e^{ 5 \pi i / 4}+ e^{\pi i / 4}
\\& \quad\qquad
+ e^{ 11 \pi i / 4}+ e^{ 7 \pi i / 4}+ e^{ 15 \pi i / 4}\big)\delta_{5} 
\\& \qquad  +  \big(  i + 1+ 1+ 1- i+ e^{ \pi  i}+  i 
\\& \quad\qquad
- i+ e^{ 3 \pi i}+ e^{ 5 \pi i / 2}- i+ e^{ 3 \pi i}\big)\delta_{6}
 \\& \qquad  +  \big( e^{ 7 \pi i / 4}+ e^{ 9 \pi i / 4}+ e^{ 9 \pi i / 4}+ e^{ 3 \pi i / 4}+ e^{ 11 \pi i / 4}+ e^{ 5 \pi i / 4}
\\& \quad\qquad
+ e^{\pi i / 4}+ e^{ 7 \pi i / 4}+ e^{ 7 \pi i / 4}+ e^{ 13 \pi i / 4}+ e^{ 5 \pi i / 4}+ e^{ 11 \pi i / 4}\big)\delta_{7} 
 \\& \qquad  +  \big( e^{ \pi  i}+ e^{ 7 \pi i / 2}+ e^{ 5 \pi i / 2}+ e^{ \pi  i}- i
\\& \quad\qquad
+ e^{ 5 \pi i / 2}+ e^{ \pi  i}- i+ 1+ 1+ 1+ e^{ 5 \pi i / 2}\big)\delta_{8} 
 \\& \qquad  +  \big( e^{ 3 \pi i / 4}+ e^{ 11 \pi i / 4}+ e^{ 15 \pi i / 4}+ e^{ 5 \pi i / 4}+ e^{ 7 \pi i / 4}+ e^{ 5 \pi i / 4}
\\& \quad\qquad
+ e^{ 9 \pi i / 4}+ e^{\pi i / 4}+ e^{ 7 \pi i / 4}+ e^{ 9 \pi i / 4}+ e^{ 3 \pi i / 4}+ e^{ 13 \pi i / 4}\big)\delta_{9} 
 \\& \qquad  +  \big(  i + e^{ 5 \pi i / 2}+ e^{ 3 \pi i}+ e^{ 5 \pi i / 2}+ 1- i+ e^{ \pi  i}
\\& \quad\qquad
- i- i+ 1+ e^{ \pi  i}+ 1\big)\delta_{10} 
 \\& \qquad  +  \big( e^{ 5 \pi i / 4}+ e^{ 9 \pi i / 4}+ e^{ 11 \pi i / 4}+ e^{ 7 \pi i / 4}+ e^{ 13 \pi i / 4}+ e^{ 7 \pi i / 4}
\\& \quad\qquad
+ e^{ 5 \pi i / 4}+ e^{\pi i / 4}+ e^{ 11 \pi i / 4}+ e^{ 7 \pi i / 4}+ e^{ 3 \pi i / 4}+ e^{ 9 \pi i / 4}\big)\delta_{11} 
 \\& \qquad  +  \big( 1+ e^{ 3 \pi i}+ e^{ 5 \pi i / 2}- i+ e^{ 5 \pi i / 2}
\\& \quad\qquad
+ e^{ 3 \pi i}- i+  i - i+ e^{ 3 \pi i}+ 1+ 1\big)\delta_{12} 
 \\& \qquad  +  \big( e^{\pi i / 4}+ e^{ 7 \pi i / 4}+ e^{ 13 \pi i / 4}+ e^{ 5 \pi i / 4}+ e^{ 9 \pi i / 4}
+ e^{ 9 \pi i / 4}+ e^{ 11 \pi i / 4}\\& \quad\qquad+ e^{ 3 \pi i / 4}+ e^{ 7 \pi i / 4}+ e^{ 7 \pi i / 4}+ e^{ 13 \pi i / 4}+ e^{ 3 \pi i}\big)\delta_{13} \
\\&= 13\delta_0.
\qedhere
\end{align*}
\end{proof}

\subsection{Sets with 14 to 16 elements}
Except for the special case in the next subsection, the computatational time needed to search for extreme
sets with more than 13 elements is prohibitive at this writing.

\begin{prop}\label{proppp16inZ16}
The 16 element set $\{0,  1,  2,  3,  4,  5,  6,  7,  8,  9,  10,  11,  12,  13,  14,  15 \}$ is extreme in $\Zp{16}{}$. \newline \end{prop}

\begin{proof}
Let  $\mu =  \delta_{0}+\delta_{1}+i\delta_{2}+\delta_{3}+\delta_{4}+i\delta_{5}+\delta_{6}-\delta_{7}+\delta_{8}-\delta_{9}-i\delta_{10}+\delta_{11}+\delta_{12}-i\delta_{13}-\delta_{14}+\delta_{15}
$.  

Then 
\begin{align*} 
\mu * \tilde \mu & = \big(1+1+1+1+1+1+1+1+1+1+1+1+1+1+1+1\big)\delta_{0}
\\&\qquad
+\big(1+1+i-i+1+i-i-1-1-1+i+i+1-i-i-1\big)\delta_{1}
\\&\qquad
+\big(-1+1+i+1-i+i+1+i+1+1-i-1+i-i-1+i\big)\delta_{2}
\\&\qquad
+\big(i-1+i+1+1+1+1-1-i-1+i+1-1+1-1+1\big)\delta_{3}
\\&\qquad
+\big(1+i-i+1+1+i-i-1+1+i-i-1+1+i-i+1\big)\delta_{4}
\\&\qquad
+\big(1+1-1-1+1+i+1+i+1-1-1+1-1-i+1+i\big)\delta_{5}
\\&\qquad
+\big(i+1+i+i-1+i+1-1-i-1-i-i+1+i-1-1\big)\delta_{6}
\\&\qquad
+\big(-1+i+i+1+i-i+1-1+1+i-i+1-i-i+1+1\big)\delta_{7}
\\&\qquad
+\big(1-1-1+1+1-1-1-1+1-1-1+1+1-1-1-1\big)\delta_{8}
\\&\qquad
+\big(-1+1-i+i+1+i+i+1+1-1-i-i+1-i+i+1\big)\delta_{9}
\\&\qquad
+\big(1-1+i-1+i+i+1-i-1-1-i+1-i-i-1-i\big)\delta_{10}
\\&\qquad
+\big(-i+1-i+1-1-1+1-1+i+1-i+1+1-1-1+1\big)\delta_{11}
\\&\qquad
+\big(1-i+i-1+1-i+i-1+1-i+i+1+1-i+i+1\big)\delta_{12}
\\&\qquad
+\big(1+1+1+1-1+i-1-i+1-1+1-1+1-i-1-i\big)\delta_{13}
\\&\qquad
+\big(-i+1+i-i+1-i+1+1+i-1-i+i-1-i-1+1\big)\delta_{14}
\\&\qquad
+\big(1-i+i+1-i+i-1-1-1-i-i+1+i+i-1+1\big)\delta_{15}\\&=16\delta_{0 }\qedhere
\end{align*}
\end{proof}

\begin{prop}\label{prop16inZ17}
The 16 element set $\{0,  1,  2,  3,  4,  5,  6,  7,  8,  9,  10,  11,  12,  13,  14,  15 \}$ is extreme in $ \Zp{17}{}$. \newline \end{prop}

\begin{proof}
Let  $\mu =  \delta_{0}+e^{6\pii/5}\delta_{1}+e^{6\pii/5}\delta_{2}+e^{8\pii/5}\delta_{3}+e^{2\pii/5}\delta_{4}+e^{8\pii/5}\delta_{5}+\delta_{6}+e^{2\pii/5}\delta_{7}+e^{2\pii/5}\delta_{8}+\delta_{9}+e^{8\pii/5}\delta_{10}+e^{2\pii/5}\delta_{11}+e^{8\pii/5}\delta_{12}+e^{6\pii/5}\delta_{13}+e^{6\pii/5}\delta_{14}+\delta_{15}
$.  Then 
\begin{align*} 
\mu * \tilde \mu & = 16\delta_{0}
+\big(3e^{6\pii/5}+3+3e^{2\pii/5}+3e^{4\pii/5}+3e^{8\pii/5}\big)
\big(\delta_{1}
+\delta_{2}+\cdots+\delta_{16}\big)\\&=16\delta_{0 }\qedhere
\end{align*}
\end{proof}

\begin{prop}\label{prop16inZ16}
The 16 element set $\{0,  1,  2,  3,  4,  5,  6,  7,  8,  9,  10,  11,  12,  13,  14,  15 \}$ is extreme in $ \times  \times  \times  \times  \times  \times  \times  \times  \times  \times  \times  \times  \times \Zp{16}{}$. \newline \end{prop}

\begin{proof}
Let  $\mu =  \delta_{0}+\delta_{1}+\delta_{2}-i\delta_{3}-\delta_{4}+i\delta_{5}+\delta_{6}-\delta_{7}+\delta_{8}-\delta_{9}+\delta_{10}-i\delta_{11}-\delta_{12}+i\delta_{13}+\delta_{14}+\delta_{15}
$.  Then 
\begin{align*} 
\mu * \tilde \mu & = \big(1+1+1+1+1+1+1+1+1+1+1+1+1+1+1+1\big)\delta_{0}
\\&\qquad
+\big(1+1+1-i-i-i-i-1-1-1-1-i-i-i-i+1\big)\delta_{1}
\\&\qquad
+\big(1+1+1-i-1-1-1+i+1+1+1+i-1-1-1-i\big)\delta_{2}
\\&\qquad
+\big(-i+1+1-i-1+i+i+1-i-1-1-i+1+i+i-1\big)\delta_{3}
\\&\qquad
+\big(-1-i+1-i-1+i+1-i-1+i+1+i-1-i+1+i\big)\delta_{4}
\\&\qquad
+\big(i-1-i-i-1+i+1-1+i+1-i-i+1+i-1+1\big)\delta_{5}
\\&\qquad
+\big(1+i-1-1-1+i+1-1+1-i-1-1-1-i+1-1\big)\delta_{6}
\\&\qquad
+\big(-1+1+i+i+i+i+1-1+1-1+i+i+i+i-1+1\big)\delta_{7}
\\&\qquad
+\big(1-1+1+1+1+1+1-1+1-1+1+1+1+1+1-1\big)\delta_{8}
\\&\qquad
+\big(-1+1-1-i-i-i-i-1+1-1+1-i-i-i-i+1\big)\delta_{9}
\\&\qquad
+\big(1-1+1+i-1-1-1+i+1-1+1-i-1-1-1-i\big)\delta_{10}
\\&\qquad
+\big(-i+1-1-i+1+i+i+1-i-1+1-i-1+i+i-1\big)\delta_{11}
\\&\qquad
+\big(-1-i+1+i-1-i+1-i-1+i+1-i-1+i+1+i\big)\delta_{12}
\\&\qquad
+\big(i-1-i-i+1+i-1-1+i+1-i-i-1+i+1+1\big)\delta_{13}
\\&\qquad
+\big(1+i-1-1-1-i+1+1+1-i-1-1-1+i+1+1\big)\delta_{14}
\\&\qquad
+\big(1+1+i+i+i+i-1-1-1-1+i+i+i+i+1+1\big)\delta_{15}\\&=16\delta_{0 }\qedhere
\end{align*}
\end{proof}

\subsection{Sets with 17 elements}

\begin{prop}\label{prop17in18}
$\{0, 1, 2, 3, 4, 5, 6, 7, 8, 9, 10, 11, 12, 13, 14, 15, 16\}$ is extreme in \Zp{18}{}.
\end{prop}
\begin{proof}
Let  $\mu =  \delta_{(0,0)}-i\delta_{(1,0)}-\delta_{(2,0)}-i\delta_{(3,0)}+\delta_{(4,0)}-i\delta_{(5,0)}-\delta_{(6,0)}+i\delta_{(7,0)}-\delta_{(8,0)}+i\delta_{(9,0)}-\delta_{(10,0)}-i\delta_{(11,0)}+\delta_{(12,0)}-i\delta_{(13,0)}-\delta_{(14,0)}-i\delta_{(15,0)}+\delta_{(16,0)}
$.  Then 
\begin{align*} 
\mu * \tilde \mu & = \big(1+1+1+1+1+1+1+1+1+1+1+1+1+1+1+1+1\big)\delta_{(0,0)}
\\&\qquad
+\big(-i-i+i+i-i-i-i+i-i+i+i+i-i-i+i+i
\\&\qquad\qquad+1-1+1-1+1-1-1+1+1+1-1-1+1-1+1-1
\\&\qquad\qquad+i-i-i+i+i-i+i-i-i+i+i-i+i-i-i+i
\\&\qquad\qquad-1+1-1+1+1+1-1-1-1+1-1-1-1+1+1+1
\\&\qquad\qquad+i+i-i-i-i-i-i-i+i-i+i-i+i+i+i+i
\\&\qquad\qquad+1+1+1+1+1-1-1+1-1-1+1-1-1+1-1-1
\\&\qquad\qquad+i-i-i+i+i-i+i-i-i-i-i+i+i+i+i-i
\\&\qquad\qquad-1+1-1+1-1+1-1-1-1+1+1+1+1+1-1-1
\\&\qquad\qquad-i+i-i-i+i+i-i+i+i-i+i+i-i-i+i-i
\\&\qquad\qquad-1-1+1+1+1+1+1-1-1-1+1-1+1-1+1-1
\\&\qquad\qquad-i+i+i+i+i-i-i-i-i+i-i+i+i-i-i+i
\\&\qquad\qquad-1-1+1-1-1+1-1-1+1-1-1+1+1+1+1+1\big)\delta_{(12,0)}
\\&\qquad
+\big(i+i+i+i-i+i-i+i-i-i-i-i-i-i+i+i\big)\delta_{(13,0)}
\\&\qquad
+\big(1+1+1-1-1-1+1-1-1-1+1+1+1-1+1-1\big)\delta_{(14,0)}
\\&\qquad
+\big(i-i-i+i-i+i+i-i-i+i-i+i+i-i-i+i\big)\delta_{(15,0)}
\\&\qquad
+\big(-1+1-1+1-1-1+1+1+1-1-1+1-1+1-1+1\big)\delta_{(16,0)}
\\&\qquad
+\big(i+i-i-i+i+i+i-i+i-i-i-i+i+i-i-i\big)\delta_{(17,0)}\\&=17\delta_{(0 , 0)}\qedhere
\end{align*}
\end{proof}

\begin{rem} \label{remRealExtremtransform}
We note that if $\mu$ above is multiplied by $\delta_{-8}$, the resulting measure is
self-adjoint (that is, $(\delta_{-8}*\mu)\widetilde{\ \ } =\delta_{-8}*\mu$). Hence, the set 
$\{-8,\dots,8\}$ is the support of an extreme measure with real transform. See 
\propref{propExtrLessSingleton} for a related result.
\end{rem}
\newpage
\section{Proofs of extremality for non-cyclic groups}\label{secnoncyclicgps}

\subsection{\Zp23}
\begin{prop}
\cite{MR627683}
\label{prop5inZp2.Zp2.Zp2} The 5-element set 
$\{(0, 0, 0), (0, 0, 1), (0, 1, 0), (0, 1, 1), (1, 0, 0) \}$ 
is extreme in \Zp{2}{3}.
\end{prop}

\begin{proof}
Let  $\mu =  \delta_{(0,0,0)}-\delta_{(0,0,1)}-\delta_{(0,1,0)}-\delta_{(0,1,1)}+i\delta_{(1,0,0)}
$.  Then 
\begin{align*} 
\mu * \tilde \mu & = \big(1+1+1+1+1\big)\delta_{(0,0,0)}
\\&\qquad
+\big(-1-1+1+1\big)\delta_{(0,0,1)}
\\&\qquad
+\big(-1+1-1+1\big)\delta_{(0,1,0)}
\\&\qquad
+\big(-1+1+1-1\big)\delta_{(0,1,1)}
\\&\qquad
+\big(-i+i\big)\delta_{(1,0,0)}
\\&\qquad
+\big(i-i\big)\delta_{(1,0,1)}
\\&\qquad
+\big(i-i\big)\delta_{(1,1,0)}
\\&\qquad
+\big(i-i\big)\delta_{(1,1,1)}\\&=5\delta_{(0 , 0, 0)}\qedhere
\end{align*}
\end{proof}

\begin{prop}\label{prop6inZ2.X.Z2.X.Z2}
The 6 element set $\{(0, 0, 0),  (0, 0, 1),  (0, 1, 0),  (0, 1, 1),  (1, 0, 0),  (1, 0, 1) \}$ is extreme in $\Zp{2}{3}$. \newline \end{prop}

 \begin{proof}
Let  $\mu =  \delta_{(0,0,0)}-i\delta_{(0,0,1)}+e^{7\pii/4}\delta_{(0,1,0)}+e^{3\pii/4}\delta_{(0,1,1)}+e^{\pii/4}\delta_{(1,0,0)}+e^{\pii/4}\delta_{(1,0,1)}
$.  Then 
\begin{align*} 
\mu * \tilde \mu & = \big(1+1+1+1+1+1\big)\delta_{(0,0,0)}
\\&\qquad
+\big(i-i-1-1+1+1\big)\delta_{(0,0,1)}
\\&\qquad
+\big(e^{\pii/4}+e^{3\pii/4}+e^{7\pii/4}+e^{5\pii/4}\big)\delta_{(0,1,0)}
\\&\qquad
+\big(e^{5\pii/4}+e^{7\pii/4}+e^{\pii/4}+e^{3\pii/4}\big)\delta_{(0,1,1)}
\\&\qquad
+\big(e^{7\pii/4}+e^{5\pii/4}+e^{\pii/4}+e^{3\pii/4}\big)\delta_{(1,0,0)}
\\&\qquad
+\big(e^{7\pii/4}+e^{5\pii/4}+e^{3\pii/4}+e^{\pii/4}\big)\delta_{(1,0,1)}
\\&\qquad
+\big(-i+i+i-i\big)\delta_{(1,1,0)}
\\&\qquad
+\big(-i+i-i+i\big)\delta_{(1,1,1)}\\&=6\delta_{(0 , 0, 0)}\qedhere
\end{align*}
\end{proof}

For \Zp23, a seven-element set has PSC   at most 2.777127870 unlike  seven-element subsets of \Zp8.

\subsection{\Zp32}

\begin{prop}\label{prop7inZp3.Zp3}
Each of the two 7-element sets below is extreme in \Zp{ 3}{2}.
\begin{enumerate}
\item  $\{(0, 0),  (0, 1),  (0, 2),  (1, 0),  (1, 1),  (2, 0),  (2, 2) \}$ 
\item $\{(0, 0),  (0, 1),  (0, 2),  (1, 0),  (1, 1),  (1, 2),  (2, 0) \}$
\end{enumerate} 

\end{prop}

\begin{proof}
(1). Let $\mu=\delta_{(0,0)}+e^{5\pii/3}\delta_{(0,1)}+e^{2\pii/3}\delta_{(0,2)}+e^{5\pii/3}\delta_{(1,0)}+\delta_{(1,1)}+e^{2\pii/3}\delta_{(2,0)}+e^{\pii/3}\delta_{(2,2)}
$.  Then 
\begin{align*} 
\mu * \tilde \mu & = \big(1+1+1+1+1+1+1\big)\delta_{(0,0)}
\\&
\quad+\big(e^{4\pii/3}+e^{5\pii/3}+e^{\pii/3}+e^{\pii/3}\big)\delta_{(0,1)}
\\&
\quad+\big(e^{\pii/3}+e^{2\pii/3}+e^{5\pii/3}+e^{5\pii/3}\big)\delta_{(0,2)}
\\&
\quad+\big(e^{4\pii/3}+e^{\pii/3}+e^{5\pii/3}+e^{\pii/3}-\big)\delta_{(1,0)}
\\&
\quad+\big(e^{5\pii/3}+e^{\pii/3}\big)\delta_{(1,1)}
\\&
\quad+\big(e^{4\pii/3}+e^{4\pii/3}+e^{2\pii/3}+e^{2\pii/3}\big)\delta_{(1,2)}
\\&
\quad+\big(e^{\pii/3}+e^{5\pii/3}+e^{2\pii/3}+e^{5\pii/3}\big)\delta_{(2,0)}
\\&
\quad+\big(e^{2\pii/3}+e^{4\pii/3}+e^{4\pii/3}+e^{2\pii/3}\big)\delta_{(2,1)}
\\&
\quad+\big(e^{5\pii/3}+e^{\pii/3}\big)\delta_{(2,2)}\\&=7\delta_{(0 , 0)}.
\end{align*}

(2). Let
$\mu=\delta_{(0,0)}+e^{4\pii/3}\delta_{(0,1)}+e^{4\pii/3}\delta_{(0,2)}+e^{5\pii/3}\delta_{(1,0)}+e^{\pii/3}\delta_{(1,1)}+e^{\pii/3}\delta_{(1,2)}+\delta_{(2,0)}
$.  Then 
\begin{align*} 
\mu * \tilde \mu & = \big(1+1+1+1+1+1+1\big)\delta_{(0,0)}
\\&
\quad+\big(e^{2\pii/3}+e^{4\pii/3}+1+e^{4\pii/3}+e^{2\pii/3}+1\big)\delta_{(0,1)}
\\&
\quad+\big(e^{2\pii/3}+1+e^{4\pii/3}+e^{4\pii/3}+1+e^{2\pii/3}\big)\delta_{(0,2)}
\\&
\quad+\big(1+e^{5\pii/3}-1-1+e^{\pii/3}\big)\delta_{(1,0)}
\\&
\quad+\big(e^{4\pii/3}+e^{\pii/3}+e^{\pii/3}-1+e^{5\pii/3}\big)\delta_{(1,1)}
\\&
\quad+\big(e^{4\pii/3}+e^{\pii/3}-1+e^{\pii/3}+e^{5\pii/3}\big)\delta_{(1,2)}
\\&
\quad+\big(e^{\pii/3}-1-1+e^{5\pii/3}+1\big)\delta_{(2,0)}
\\&
\quad+\big(e^{5\pii/3}+e^{5\pii/3}-1+e^{\pii/3}+e^{2\pii/3}\big)\delta_{(2,1)}
\\&
\quad+\big(e^{5\pii/3}-1+e^{5\pii/3}+e^{\pii/3}+e^{2\pii/3}\big)\delta_{(2,2)}\\&=7\delta_{(0 , 0)}.
\end{align*}
 \qedhere
\end{proof}

\begin{prop}\label{prop7inZ3.X.Z3}
The 7 element set $\{(0, 0),  (0, 1),  (0, 2),  (1, 0),  (1, 1),  (1, 2),  (2, 0) \}$ is extreme in $\Zp{3}{2}$. \newline \end{prop}

\begin{proof}
Let  $\mu =  \delta_{(0,0)}+e^{4\pii/3}\delta_{(0,1)}+e^{4\pii/3}\delta_{(0,2)}+e^{\pii/3}\delta_{(1,0)}-\delta_{(1,1)}-\delta_{(1,2)}+e^{4\pii/3}\delta_{(2,0)}
$.  Then 
\begin{align*} 
\mu * \tilde \mu & = \big(1+1+1+1+1+1+1\big)\delta_{(0,0)}
\\&\qquad
+\big(e^{2\pii/3}+e^{4\pii/3}+1+e^{4\pii/3}+e^{2\pii/3}+1\big)\delta_{(0,1)}
\\&\qquad
+\big(e^{2\pii/3}+1+e^{4\pii/3}+e^{4\pii/3}+1+e^{2\pii/3}\big)\delta_{(0,2)}
\\&\qquad
+\big(e^{2\pii/3}+e^{\pii/3}+e^{5\pii/3}+e^{5\pii/3}-1\big)\delta_{(1,0)}
\\&\qquad
+\big(1-1-1+e^{5\pii/3}+e^{\pii/3}\big)\delta_{(1,1)}
\\&\qquad
+\big(1-1+e^{5\pii/3}-1+e^{\pii/3}\big)\delta_{(1,2)}
\\&\qquad
+\big(e^{5\pii/3}+e^{\pii/3}+e^{\pii/3}-1+e^{4\pii/3}\big)\delta_{(2,0)}
\\&\qquad
+\big(-1-1+e^{\pii/3}+e^{5\pii/3}+1\big)\delta_{(2,1)}
\\&\qquad
+\big(-1+e^{\pii/3}-1+e^{5\pii/3}+1\big)\delta_{(2,2)}\\&=7\delta_{(0 , 0)}\qedhere
\end{align*}
\end{proof}


\subsection{\Zp{4}{2}.}

The following is a variant of 
\cite[3.3 (i) ]{MR627683}.

\begin{prop}\label{prop5inZ4.X.Z4}
Each of the following  5-element sets  is extreme in $  \Zp{4}{2}$. 
\begin{enumerate}
\item $\{(0, 0),  (0, 1),  (0, 2),  (0, 3),  (2, 0) \}$
\item $\{(0, 0),  (0, 1),  (0, 2),  (2, 0),  (2, 2) \}$
\end{enumerate}
 \end{prop}

\begin{proof} (1).
Let  $\mu =  \delta_{(0,0)}-i\delta_{(0,1)}+\delta_{(0,2)}+i\delta_{(0,3)}+i\delta_{(2,0)}
$.  Then 
\begin{align*} 
\mu * \tilde \mu & = \big(1+1+1+1+1\big)\delta_{(0,0)}
+\big(-i-i+i+i\big)\delta_{(0,1)}
\\&\qquad
+\big(1-1+1-1\big)\delta_{(0,2)}
\\&\qquad
+\big(i-i-i+i\big)\delta_{(0,3)}
+\big(-i+i\big)\delta_{(2,0)}
\\&\qquad
+\big(-1+1\big)\delta_{(2,1)}
+\big(-i+i\big)\delta_{(2,2)}
+\big(1-1\big)\delta_{(2,3)}\\&=5\delta_{(0 , 0)}.
\end{align*}
\

(2).
Let  $\mu =  \delta_{(0,0)}-\delta_{(0,1)}-\delta_{(0,2)}+i\delta_{(2,0)}+i\delta_{(2,2)}
$.  Then 
\begin{align*} 
\mu * \tilde \mu & = \big(1+1+1+1+1\big)\delta_{(0,0)}
+\big(-1+1\big)\delta_{(0,1)}
\\&\qquad
+\big(-1-1+1+1\big)\delta_{(0,2)}
+\big(-1+1\big)\delta_{(0,3)}
\\&\qquad
+\big(-i+i+i-i\big)\delta_{(2,0)}
+\big(i-i\big)\delta_{(2,1)}
\\&\qquad
+\big(-i+i-i+i\big)\delta_{(2,2)}
+\big(i-i\big)\delta_{(2,3)}\\&=5\delta_{(0 , 0)}\qedhere
\end{align*}
\end{proof}

\begin{prop}\label{prop6inZ4.X.Z4}
Each of the 6-element sets 
\begin{enumerate}
\item $\{(0, 0),  (0, 1),  (0, 2),  (0, 3),  (1, 0),  (1, 2) \}$
\item $\{(0, 0),  (0, 1),  (0, 2),  (1, 0),  (1, 1),  (1, 2) \}$
\end{enumerate}
is extreme in $\Zp2{} \times \Zp{4}{}$.  
\end{prop} 

\begin{rem}
The 2018/10/30 version of my program (findNonCyclic) did not find those two sets equivalent. A better program might, however.
\end{rem}

\begin{proof} (1). 
Let  $\mu =  \delta_{(0,0)}+e^{7\pii/4}\delta_{(0,1)}+i\delta_{(0,2)}+e^{3\pii/4}\delta_{(0,3)}+e^{3\pii/4}\delta_{(1,0)}+e^{3\pii/4}\delta_{(1,2)}
$.  Then 
\begin{align*} 
\mu * \tilde \mu & = \big(1+1+1+1+1+1\big)\delta_{(0,0)}
\\&\qquad
+\big(e^{5\pii/4}+e^{7\pii/4}+e^{3\pii/4}+e^{\pii/4}\big)\delta_{(0,1)}
\\&\qquad
+\big(-i-1+i-1+1+1\big)\delta_{(0,2)}
\\&\qquad
+\big(e^{\pii/4}+e^{5\pii/4}+e^{7\pii/4}+e^{3\pii/4}\big)\delta_{(0,3)}
\\&\qquad
+\big(e^{5\pii/4}+e^{7\pii/4}+e^{3\pii/4}+e^{\pii/4}\big)\delta_{(1,0)}
\\&\qquad
+\big(-1+1+1-1\big)\delta_{(1,1)}
\\&\qquad
+\big(e^{5\pii/4}+e^{7\pii/4}+e^{\pii/4}+e^{3\pii/4}\big)\delta_{(1,2)}
\\&\qquad
+\big(-1+1-1+1\big)\delta_{(1,3)}\\&=6\delta_{(0 , 0)}\qedhere
\end{align*}

(2). 
Let  $\mu =  \delta_{(0,0)}+e^{23\pii/12}\delta_{(0,1)}-i\delta_{(0,2)}+e^{\pii/3}\delta_{(1,0)}+e^{11\pii/12}\delta_{(1,1)}+e^{11\pii/6}\delta_{(1,2)}
$.  Then 
\begin{align*} 
\mu * \tilde \mu & = \big(1+1+1+1+1+1\big)\delta_{(0,0)}
\\&\qquad
+\big(e^{23\pii/12}+e^{19\pii/12}+e^{7\pii/12}+e^{11\pii/12}\big)\delta_{(0,1)}
\\&\qquad
+\big(i-i+i-i\big)\delta_{(0,2)}
\\&\qquad
+\big(e^{\pii/12}+e^{5\pii/12}+e^{17\pii/12}+e^{13\pii/12}\big)\delta_{(0,3)}
\\&\qquad
+\big(e^{5\pii/3}-1+e^{5\pii/3}+e^{\pii/3}-1+e^{\pii/3}\big)\delta_{(1,0)}
\\&\qquad
+\big(e^{19\pii/12}+e^{7\pii/12}+e^{11\pii/12}+e^{23\pii/12}\big)\delta_{(1,1)}
\\&\qquad
+\big(e^{\pii/6}+e^{7\pii/6}+e^{5\pii/6}+e^{11\pii/6}\big)\delta_{(1,2)}
\\&\qquad
+\big(e^{13\pii/12}+e^{\pii/12}+e^{5\pii/12}+e^{17\pii/12}\big)\delta_{(1,3)}\\&=6\delta_{(0 , 0)}\qedhere
\end{align*}
\end{proof}

\subsection{\Zp43}
These groups have 25 or more elements, too many to find equivalence classes for sets with 6 or more elements in a
reasonable time.


\subsection{$\Zp{2}{}\times\Zp4{}$}

\begin{prop}\label{prop4inZ2.X.Z4}
The 4 element set $\{(0, 0),  (0, 1),  (1, 3),  (1, 0) \}$ is extreme in $\Zp{2}{} \times \Zp{4}{}$. \newline \end{prop}

\begin{proof}
Let  $\mu =  \delta_{(0,0)}+i\delta_{(0,1)}-\delta_{(1,3)}+i\delta_{(1,0)}
$.  Then 
\begin{align*} 
\mu * \tilde \mu & = \big(1+1+1+1\big)\delta_{(0,0)}
\\&\qquad
+\big(i-i\big)\delta_{(0,1)}
\\&\qquad
+\big(-i+i\big)\delta_{(0,3)}
\\&\qquad
+\big(-i+i\big)\delta_{(1,0)}
\\&\qquad
+\big(-1+1\big)\delta_{(1,1)}
\\&\qquad
+\big(-i+i\big)\delta_{(1,2)}
\\&\qquad
+\big(-1+1\big)\delta_{(1,3)}\\&=4\delta_{(0 , 0)}\qedhere
\end{align*}
\end{proof}

\begin{prop}\label{prop5inZ2.X.Z4}
Each of the following 5-element sets is extreme in \Zp2{}\Times\Zp4{}.
\begin{enumerate}

\item $\{(0, 0),  (0, 1),  (0, 2),  (1, 0),  (1, 2) \}$
\item $\{(0, 0),  (0, 1),  (0, 2),  (0, 3),  (1, 0) \}$ 
\end{enumerate}
 \end{prop} 
 
\begin{proof}
(1). 
Let  $\mu =  \delta_{(0,0)}-\delta_{(0,1)}-\delta_{(0,2)}+i\delta_{(1,0)}+i\delta_{(1,2)}
$.  Then 
\begin{align*} 
\mu * \tilde \mu & = \big(1+1+1+1+1\big)\delta_{(0,0)}
+\big(-1+1\big)\delta_{(0,1)}
+\big(-1-1+1+1\big)\delta_{(0,2)}
\\&\qquad
+\big(-1+1\big)\delta_{(0,3)}
+\big(-i+i+i-i\big)\delta_{(1,0)}
+\big(i-i\big)\delta_{(1,1)}
\\&\qquad
+\big(-i+i-i+i\big)\delta_{(1,2)}
+\big(i-i\big)\delta_{(1,3)}\\&=5\delta_{(0 , 0)}.
\end{align*}
 
 (2). 
Let  $\mu =  \delta_{(0,0)}-i\delta_{(0,1)}+\delta_{(0,2)}+i\delta_{(0,3)}+i\delta_{(1,0)}
$.  Then 
\begin{align*} 
\mu * \tilde \mu & = \big(1+1+1+1+1\big)\delta_{(0,0)}
+\big(-i-i+i+i\big)\delta_{(0,1)}
\\&\qquad
+\big(1-1+1-1\big)\delta_{(0,2)}
+\big(i-i-i+i\big)\delta_{(0,3)}
+\big(-i+i\big)\delta_{(1,0)}
\\&\qquad
+\big(-1+1\big)\delta_{(1,1)}
+\big(-i+i\big)\delta_{(1,2)}
+\big(1-1\big)\delta_{(1,3)}\\&=5\delta_{(0 , 0)}\qedhere
\end{align*}
\end{proof}

\begin{prop}\label{prop6inZ2.X.Z4}
Each of the following 6-element sets  
is  extreme in $\Zp{2}{} \times \Zp{4}{}$. 
\begin{enumerate}
\item$\{(0, 0),  (0, 1),  (0, 2),  (0, 3),  (1, 0),  (1, 2) \}$
\item  $\{(0, 0),  (0, 1),  (0, 2),  (1, 0),  (1, 1),  (1, 2) \}$
\end{enumerate}
 \end{prop}

\begin{proof}
(1.)
Let  $\mu =  \delta_{(0,0)}+e^{7\pii/4}\delta_{(0,1)}+i\delta_{(0,2)}+e^{3\pii/4}\delta_{(0,3)}+e^{3\pii/4}\delta_{(1,0)}+e^{3\pii/4}\delta_{(1,2)}
$.  Then 
\begin{align*} 
\mu * \tilde \mu & = \big(1+1+1+1+1+1\big)\delta_{(0,0)}
\\&\qquad
+\big(e^{5\pii/4}+e^{7\pii/4}+e^{3\pii/4}+e^{\pii/4}\big)\delta_{(0,1)}
\\&\qquad
+\big(-i-1+i-1+1+1\big)\delta_{(0,2)}
\\&\qquad
+\big(e^{\pii/4}+e^{5\pii/4}+e^{7\pii/4}+e^{3\pii/4}\big)\delta_{(0,3)}
\\&\qquad
+\big(e^{5\pii/4}+e^{7\pii/4}+e^{3\pii/4}+e^{\pii/4}\big)\delta_{(1,0)}
+\big(-1+1+1-1\big)\delta_{(1,1)}
\\&\qquad
+\big(e^{5\pii/4}+e^{7\pii/4}+e^{\pii/4}+e^{3\pii/4}\big)\delta_{(1,2)}
+\big(-1+1-1+1\big)\delta_{(1,3)}\\&=6\delta_{(0 , 0)}.
\end{align*}

(2).
Let  $\mu =  \delta_{(0,0)}+e^{23\pii/12}\delta_{(0,1)}-i\delta_{(0,2)}+e^{\pii/3}\delta_{(1,0)}+e^{11\pii/12}\delta_{(1,1)}+e^{11\pii/6}\delta_{(1,2)}
$.  Then 
\begin{align*} 
\mu * \tilde \mu & = \big(1+1+1+1+1+1\big)\delta_{(0,0)}
\\&\qquad
+\big(e^{23\pii/12}+e^{19\pii/12}+e^{7\pii/12}+e^{11\pii/12}\big)\delta_{(0,1)}
\\&\qquad
+\big(i-i+i-i\big)\delta_{(0,2)}
\\&\qquad
+\big(e^{\pii/12}+e^{5\pii/12}+e^{17\pii/12}+e^{13\pii/12}\big)\delta_{(0,3)}
\\&\qquad
+\big(e^{5\pii/3}-1+e^{5\pii/3}+e^{\pii/3}-1+e^{\pii/3}\big)\delta_{(1,0)}
\\&\qquad
+\big(e^{19\pii/12}+e^{7\pii/12}+e^{11\pii/12}+e^{23\pii/12}\big)\delta_{(1,1)}
\\&\qquad
+\big(e^{\pii/6}+e^{7\pii/6}+e^{5\pii/6}+e^{11\pii/6}\big)\delta_{(1,2)}
\\&\qquad
+\big(e^{13\pii/12}+e^{\pii/12}+e^{5\pii/12}+e^{17\pii/12}\big)\delta_{(1,3)}\\&=6\delta_{(0 , 0)}\qedhere
\end{align*}
\end{proof}

\subsection{\Zp{2}{2}\Times\Zp{3}{}}

\begin{prop}\label{prop8inZ2.X.Z2.X.Z3}
The 8 element set 
\[
\{(0, 0, 0),  (0, 0, 1),  (0, 1, 0),  (0, 1, 1),  (1, 0, 0),  (1, 0, 1),  (1, 1, 0),  (1, 1, 1) \}
\] 
is extreme in $\Zp{2}{2}\Zp{3}{}$. \newline \end{prop}

\begin{proof}
Let  $\mu =  \delta_{(0,0,0)}+e^{7\pii/15}\delta_{(0,0,1)}-i\delta_{(0,1,0)}+e^{29\pii/30}\delta_{(0,1,1)}+i\delta_{(1,0,0)}+e^{29\pii/30}\delta_{(1,0,1)}-\delta_{(1,1,0)}+e^{7\pii/15}\delta_{(1,1,1)}
$.  Then 
\begin{align*} 
\mu * \tilde \mu & = \big(1+1+1+1+1+1+1+1\big)\delta_{(0,0,0)}
\\&\qquad
+\big(e^{7\pii/15}+e^{22\pii/15}+e^{7\pii/15}+e^{22\pii/15}\big)\delta_{(0,0,1)}
\\&\qquad
+\big(e^{23\pii/15}+e^{8\pii/15}+e^{23\pii/15}+e^{8\pii/15}\big)\delta_{(0,0,2)}
\\&\qquad
+\big(i-i-i+i-i+i+i-i\big)\delta_{(0,1,0)}
\\&\qquad
+\big(e^{29\pii/30}+e^{29\pii/30}+e^{59\pii/30}+e^{59\pii/30}\big)\delta_{(0,1,1)}
\\&\qquad
+\big(e^{31\pii/30}+e^{31\pii/30}+e^{\pii/30}+e^{\pii/30}\big)\delta_{(0,1,2)}
\\&\qquad
+\big(-i-i+i+i+i+i-i-i\big)\delta_{(1,0,0)}
\\&\qquad
+\big(e^{59\pii/30}+e^{59\pii/30}+e^{29\pii/30}+e^{29\pii/30}\big)\delta_{(1,0,1)}
\\&\qquad
+\big(e^{31\pii/30}+e^{31\pii/30}+e^{\pii/30}+e^{\pii/30}\big)\delta_{(1,0,2)}
\\&\qquad
+\big(-1+1-1+1-1+1-1+1\big)\delta_{(1,1,0)}
\\&\qquad
+\big(e^{22\pii/15}+e^{7\pii/15}+e^{22\pii/15}+e^{7\pii/15}\big)\delta_{(1,1,1)}
\\&\qquad
+\big(e^{23\pii/15}+e^{8\pii/15}+e^{23\pii/15}+e^{8\pii/15}\big)\delta_{(1,1,2)}\\&=8\delta_{(0 , 0, 0)}\qedhere
\end{align*}
\end{proof}

\bibliographystyle{amsplain} 
\bibliography{ccg}

\newpage
\section{Disclaimers} 
{\Small
\subsection*{\ \ A request}  Suggestions for things to look at, programming ideas, and  corrections to the text are invited.

\subsection*{\ \  Conjectures} I have mostly used ``conjecture'' rather than ``problem'' or ``question,'' in the hope
that ``conjecture'' will pique competitive instincts in a way that the other terms may not. 

\subsection*{\ \ To do - theory}
For  each   conjecture, show that it is correct (or give counter example(s) to it). In some cases this will
surely involve rephrasing of the conjecture.
\smallbreak

Criteria for ``exceptional'' extreme sets in non-cyclic groups need formulating. 

\subsection*{\ \ To do - computation} 
There is much computer work to be done for larger sets in larger  
groups, where the results are limited by the slowness of the computer and my
meagre programming skills.

\smallbreak
 A related project is to produce lists of extremal measures for each extreme set, in the hope of revealing  a
pattern (or patterns) that would enable solution of one or more of the conjectures. 
This would take relatively minor adaptations of the present programs.

\smallbreak
 Another project is to compile a list of sets with (close to) integer PSCs and then verify that the PSCs of
those sets are indeed integers.

\subsection*{\ \ Too much detail}
I don't expect anyone to read everything, 
  but the data is here if for those who want (excessively) repetitive exercises with answers.
  
  Also, some of what's here duplicates, partially, material of 
  the two papers in the bibliography of which I am an author or co-author. 
  
\subsection*{\ \ Are sets actually equivalent?}  Yes, for cyclic groups.  Perhaps not, for product groups: implementing 
in product groups a correct method for  finding true equivalence classes is 
 complicated and may well have escaped me; see, e.g., 
\cite{MR2363058}. Any help on this point would be much appreciated.

\section{\ \ The computer programs} I will be happy to send my computer programs to anyone who wants them.

\bibliographystyle{amsplain} 
\bibliography{ccg}

\end{document}